\documentclass[preprint,11pt]{article}

\usepackage[english]{babel}	
\usepackage[utf8x]{inputenc}		
\usepackage{enumerate}
\usepackage{textcomp}                
\usepackage{typearea}
\usepackage{pdfpages}
\usepackage[toc]{appendix}
\usepackage[disable]{todonotes}
\usepackage{subcaption}
\usepackage{mathtools}

\usepackage{amsfonts, amsmath, amssymb, amsthm, dsfont}
\usepackage{amsthm}			
\usepackage{thumbpdf}	
\usepackage{natbib}
\usepackage{booktabs}

\usepackage{pgf, tikz}
\usepackage{tikz-cd}		
\usepackage{bbm}

\theoremstyle{plain}
\newtheorem{definition}{Definition}[section]

\newtheorem*{assumption*}{Assumption}

\newtheorem*{condition*}{Condition}
\newtheorem{example}{Example}
\theoremstyle{plain}
\newtheorem{theorem}[definition]{Theorem}
\newtheorem{proposition}[definition]{Proposition}
\newtheorem{lemma}[definition]{Lemma}

\theoremstyle{remark}
\newtheorem{remark}{Remark}

\usepackage{graphicx}
\usepackage{hyperref}
\allowdisplaybreaks

\newcommand{\N}{\mathbb{N}}
\newcommand{\Z}{\mathbb{Z}}

\newcommand{\R}{\mathbb{R}}
\newcommand{\E}{\mathbb{E}}
\newcommand{\A}{\mathcal{A}}
\newcommand{\F}{\mathcal{F}}

\newcommand{\Cov}{\operatorname{Cov}}

\newcommand{\deq}{\overset{d}{=}}

\newcommand{\tr}{\text{tr}}
\newcommand{\Y}{\mathcal{Y}}
\newcommand{\beps}{\boldsymbol{\epsilon}}

\newif\ifhideproofs
\ifhideproofs
  \usepackage{environ}
  \NewEnviron{hide}{}

\fi

\title{Sequential Gaussian approximation for nonstationary time series in high dimensions}
\author{Fabian Mies\footnote{Corresponding author,
		mies@stochastik.rwth-aachen.de}\and Ansgar Steland}
\date{RWTH Aachen University, Germany}
\begin{document}
\maketitle

\begin{abstract}
	Gaussian couplings of partial sum processes are derived for the high-dimensional regime $d=o(n^\frac{1}{3})$. 
	The coupling is derived for sums of independent random vectors and subsequently extended to nonstationary time series.
	Our inequalities depend explicitly on the dimension and on a measure of nonstationarity, and are thus also applicable to arrays of random vectors.
	To enable high-dimensional statistical inference, a feasible Gaussian approximation scheme is proposed.
	Applications to sequential testing and change-point detection are described.
	
\end{abstract}

\section{Introduction}

The multivariate central limit theorem (CLT) states that, under certain regularity conditions, the standardized sum $S_n=1/\sqrt{n}\sum_{t=1}^n X_t$ tends towards a Gaussian distribution as $n\to\infty$, where $X_t$ are $d$-dimensional random vectors.
However, if the dimension $d=d_n$ is allowed to increase with $n$, standard asymptotic theory is no longer applicable because there is no suitable limit distribution.
In this high-dimensional setting, the classical CLT may be replaced by a suitable approximating Gaussian random vector $\Y_n$, and sufficiently sharp quantitative bounds on a probability metric.
It is then of particular interest to study the dependence of these bounds on the dimension $d$.

There are conceptually two approaches to quantify the Gaussian approximation.
The first approach is to construct a coupling, i.e.\ to define $S_n$ and $\Y_n$ on the same probability space, and to consider the difference $\|S_n-\Y_n\|$, where $\|\cdot\|$ denotes the Euclidean vector norm.
This suggests to study the central limit theorem in Wasserstein metrics.
A corresponding result with optimal dependence on the dimension has recently been obtained by \cite{Zhai2018}.
For iid bounded random vectors, say $\|X_t\|\leq \beta$, on a richer probability space, he constructs a Gaussian approximation such that $\E \|S_n-\Y_n\|^2 =\mathcal{O}(\frac{d \log(n)^2 \beta^2}{n})$.
This bound has been improved by \cite{eldan2020} to $\mathcal{O}(\frac{d \log(n) \beta^2}{n})$.
Since the Euclidean norm in $d$ dimensions is typically of the order $\beta \asymp\sqrt{d}$, these results effectively impose the restriction $d = o(\sqrt{n})$.
\cite{bonis2020} studies Gaussian approximations in Wasserstein distances based on higher order moments.

The second approach is to give bounds for the errors $\sup_{A\in\A}|P(S_n\in A)-P(\Y_n\in A)|$ for certain sets $\A\subset\text{Pot}(\R^{d_n})$.
In their seminal work, \cite{Chernozhukov2013,Chernozhukov2017,chernozhukov2020} show that the dimension may grow as $d=\exp(n^c)$ for some suitable $c$, if $\A$ is the class of hyperrectangles in $\R^d$.
Their ideas have been extended to dependent random vectors by \cite{Zhang2017gauss}, \cite{Zhang2018}, and \cite{kurisu2021}, see also the supplement of \cite{Chernozhukov2019}. 
These results enable many statistical applications, e.g.\ the construction of simultaneous confidence intervals or control of the family-wise error rate in multiple testing, which may be formulated in terms of test statistics of supremum type. 
However, the restriction that $A$ is a hyperrectangle excludes statistics which are, for example, based on Euclidean norms.

Both approaches presented above are concerned with approximations for the full sum $S_n$.
In this paper, we construct a sequential coupling of the partial sum process $S_k = \frac{1}{\sqrt{n}} \sum_{t=1}^k X_t$, $k=1,\ldots,n$, and a Gaussian process $\Y_k$, such that
\begin{align}
	\max_{k\leq n} \left\| S_k - \Y_k  \right\| = \mathcal{O}_P(\tau_n). \label{eqn:intro}
\end{align}
The problem \eqref{eqn:intro} is also known as a strong approximation in the literature, and may be regarded as a quantitative analogue of Donsker's invariance principle.
If the random vectors $X_t$ have finite $q$-th moments, the optimal rate is known to be $\tau_n =\mathcal{O}(n^{\frac{1}{q}-\frac{1}{2}})$ \citep{breiman1967}. 
This rate has been achieved by \cite{komlos1975,komlos1976} for univariate iid summands, and by \cite{einmahl1989} for multivariate iid summands.
We refer to \cite{Zaitsev2013} for an extensive literature review on strong Gaussian approximations for independent random vectors.
In the dependent case, this optimal rate has recently been achieved by \cite{Berkes2014} for univariate stationary time series, and by \cite{Karmakar2020} for multivariate nonstationary time series.
Earlier results for multivariate time series are due to \cite{liu2009} for the stationary case, and \cite{Wu2011} for the nonstationary case.

The goal of this paper is to provide a strong approximation result for high-dimensional, nonstationary time series.
Although most existing work on strong approximations considers a fixed dimension, some results make the dependence on the dimension explicit and are hence applicable in a high-dimensional setting.
To the best of our knowledge, the sharpest bound in terms of the dimension is $\tau_n = n^{\frac{1}{q}-\frac{1}{2}} d^{\frac{15}{2}+\alpha}$ for small $\alpha>0$, see e.g.\ Corollary 3 in \cite{zaitsev2007}.
This term vanishes at best if $d=o(n^\frac{1}{15})$, which is rather restrictive for high-dimensional applications.
Our first contribution is to construct a strong approximation for independent summands such that $\tau_n = \sqrt{\log (n)}\, d^{\frac{3}{4}-\frac{1}{2q}} n^{\frac{1}{2q}-\frac{1}{4}}$, see Theorem \ref{thm:eldan-sequential}.
For large $p$, this result allows for a dimension growing as $d=\mathcal{O}(n^{\frac{1}{3}-\alpha})$.
On the other hand, this improved dependence on the dimension comes at the price of a suboptimal rate with respect to $n$.
To derive the high-dimensional approximation result, we combine the coupling of \cite{eldan2020} with a blocking technique suggested by \cite{csorgo1975}. 

As a second contribution, we extend the high-dimensional approximation result to nonstationary time series, using the physical dependence measure introduced by \cite{Wu2005}.
Under mild regularity conditions, we also derive an explicit expression for the approximating Gaussian process in terms of local long run covariances, and suggest a multiplier bootstrap procedure to perform statistical inference.
In contrast, the result of \cite{Karmakar2020} does not provide an explicit expression for the Gaussian variance, while \cite{Berkes2014} provide explicit expressions only for the stationary univariate case.
Crucial for our feasible bootstrap approximation is some distributional regularity in time, which we formulate in terms of a total variation norm of the underlying kernel function. 
Our results are in particular applicable to high-dimensional, locally-stationary time series.
As a technical tool, we also derive a high-dimensional Rosenthal inequality for dependent random vectors, which might be useful in other situations.

All our approximation results are formulated as valid inequalities for finite sample size $n$.
In particular, the effects of dimension, moment bounds, nonstationarity, and dependency (in the time series setting), are all stated explicitly. 
Thus, we do not need to regard the observations $X_1,\ldots, X_n$, as a finite subset of an infinite sequence of random variables, and our results may also be applied to arrays $X_{1,n},\ldots, X_{n,n}$ of random variables.
For example, in Section \ref{sec:dependent}, we use this flexibility to study a class of locally stationary time series.

The remainder of this paper is structured as follows.
The strong approximation for independent random vectors is presented and discussed in Section \ref{sec:independent}.
Section \ref{sec:dependent} describes the extension to dependent random vectors and presents the Rosenthal inequality for high-dimensional time series.
The bootstrap approximation is discussed in Section \ref{sec:bootstrap}.
All proofs and technical results are gathered in Section \ref{sec:proofs}

\paragraph*{Notation:}
For $a,b\in\R$, denote $a\wedge b=\min(a,b)$, $a\vee b=\max(a,b)$.
For $x>0$, we use $\log^*(x)=\log(x)\vee 1$. 
For $v=(v_1,\ldots, v_d)\in\R^d$, denote $\|v\|_r = (\sum_{i=1}^d |v_i|^r)^\frac{1}{r}$, and $\|v\|=\|v\|_2$ is the Euclidean norm.
For a matrix $A$, $\tr(A)$ denotes the trace of $A$, and $\|A\|_{\tr} = \tr((A^T A)^\frac{1}{2})$ is the trace norm, also known as the nuclear norm or the Schatten norm of order $1$.
If $A$ is symmetric, we denote by $\lambda_{\max}(A)$ and $\lambda_{\min}(A)$ the largest and smallest eigenvalues of $A$, respectively.
For $v\in\R^d$, the outer product is denoted as $v^{\otimes 2} = v v^T$.

\section{Strong approximation for independent random vectors}\label{sec:independent}

The results of \cite{Zhai2018} and \cite{eldan2020} for the central limit theorem in the Wasserstein metric yield Gaussian couplings for the sum $S_n = \sum_{t=1}^n X_t$, even in high-dimensional settings.
Here, we show that these couplings may also be leveraged to obtain a sequential Gaussian approximation of the partial sums $S_k = \sum_{t=1}^k X_t$, i.e.\ of the form \eqref{eqn:intro}.
To this end, we use the following coupling for the full sum, which is essentially due to \cite{eldan2020}.
We slightly extend their result to unbounded, not identically distributed random vectors.

\begin{theorem}\label{thm:wasserstein-eldan}
	Let $X_1,\ldots, X_n$ be independent, centered, $d$-variate random vectors which admit the bound 
	$(\E\|X_t\|^q)^\frac{1}{q}\leq b_t$, $t=1,\ldots, n$, for some $q>2$.
	On a different probability space, there exist independent random vectors $\tilde{X}_t\deq X_t$, and independent Gaussian random vectors $Y_t \sim\mathcal{N}(0,\Cov(X_t))$, such that for some universal $C>0$,
	\begin{align*}
		\left(\E \left\| \frac{1}{\sqrt{n}}\sum_{t=1}^n(\tilde{X}_t - Y_t) \right\|^2 \right)^\frac{1}{2}
		&\leq \frac{C}{\sqrt{q-2}\wedge 1} \left( \frac{d}{n} \right)^{\frac{1}{2}-\frac{1}{q}} \sqrt{ \log (n)}  \sqrt{\frac{1}{n}\sum_{t=1}^n b_t^{2}}.
	\end{align*}
\end{theorem}  

Letting $q\to\infty$ recovers the approximation rate of \cite{eldan2020} for bounded random vectors.
Note that all vectors are measured in the Euclidean norm $\|\cdot\|$. 
Therefore, the typical size of the moment bounds is $b_i \asymp \sqrt{d}$.
Hence, for $q$ sufficiently large, the bound vanishes if $d=o(n^{\frac{1}{2}-\alpha})$ for some $\alpha>0$.

Theorem \ref{thm:wasserstein-eldan} only provides an approximation for the full sum, but not for partial sums. 
In particular, the constructed $2d$-variate random vectors $(\tilde{X}_t, Y_t)$, $t=1,\ldots, n$, are in general not independent, such that Doob's maximal inequality is not applicable.
To derive a sequential approximation result, we pursue an idea due to \cite{csorgo1975}. 
Split the sum into $M$ blocks of size $L$, and assume $n=L \cdot M$ for simplicity. 
Denote each block by $S^j=\sum_{t=(j-1)L+1}^{jL} X_t$.
On a potentially different probability space, we may define $\tilde{X}_t\deq X_t$, $\tilde{S}^j\deq S^j$, and Gaussian random vectors $\tilde{Y}^j = \sum_{t=(j-1)L+1}^{jL} Y_t$ via Theorem \ref{thm:wasserstein-eldan} such that $\tilde{S}^j$ and $\tilde{Y}^j$ are closely coupled.
Since the blocks are independent, we find that $\tilde{S}_k = \sum_{t=1}^k \tilde{X}_t \deq S_k$, and Doob's inequality is now applicable so that the partial sum $\tilde{S}_{jL}$ is well approximated by $\sum_{r=1}^j \tilde{Y}^r = \sum_{t=1}^{jL} Y_t$.
However, this construction incurs a discretization error due to the blocking of partial sums. 
Choosing the block size suitably, we obtain the following strong approximation as the first main result of this paper. 

\begin{theorem}\label{thm:eldan-sequential}
	Let $X_1,\ldots, X_n$ be independent, centered, $d$-variate random vectors which admit the bound
	$(\E\|X_t\|^q)^\frac{1}{q}\leq B$, for some $q>2$.
	On a sufficiently rich probability space, there exist independent random vectors $\tilde{X}_t\deq X_t$, and independent Gaussian random vectors $Y_t \sim\mathcal{N}(0,\Cov(X_t))$, such that for some universal $C=C(q)$ depending on $q$ only,
	\begin{align*}
		\left(\E \max_{k=1,\ldots,n} \left\| \frac{1}{\sqrt{n}}\sum_{t=1}^k(\tilde{X}_t - Y_t) \right\|^2 \right)^\frac{1}{2}
		&\leq C B \sqrt{\log(n)} \left( \frac{d}{n} \right)^{\frac{1}{4}-\frac{1}{2q}}.
	\end{align*}
\end{theorem}

If the $d$ components of each $X_t$ are iid, then the moment bound scales as $B\asymp \sqrt{d}$, hence we may consider this as the generic scaling.
In this case, the rate of Theorem \ref{thm:eldan-sequential} is $\tau_n = \sqrt{\log (n)} d^{\frac{3}{4}-\frac{1}{2q}} n^{\frac{1}{2q}-\frac{1}{4}}$, which tends to zero as long as $d = o(n^{\frac{q-2}{3q-2}-\alpha})$ for some $\alpha>0$, i.e.\ $d=o(n^{\frac{1}{3}-\alpha})$ for large $q$.
The constraint on $d$ can be relaxed if additional sparsity is assumed, such that $B\ll \sqrt{d}$, as demonstrated by the following example.

\begin{example}
	Consider iid categorical random variables $Z_1,\ldots, Z_n$, i.e.\ the $Z_t$ take values in the finite set $\{1,\ldots,d\}$.
	A natural sequential estimator of the distribution $p_j = P(Z_1=j)$, $j=1,\ldots,d$, is given by $\hat{p}_j(u) = \frac{1}{n} \sum_{t=1}^{\lfloor un\rfloor} \mathds{1}(Z_t=j)$. 
	If we introduce the $d$-variate random vectors $X_t$ by $(X_t)_j = \mathds{1}(Z_t=j)-p_j$, $j=1,\ldots,d$, then we may write $\hat{p}(u)-u\cdot p = \frac{1}{n} \sum_{t=1}^{\lfloor un\rfloor} X_t$.
	Thus, Theorem \ref{thm:eldan-sequential} is applicable. 
	The special structure of $X_t$ yields that $\|X_t\| \leq \|X_t\|_1 \leq 2$.
	In particular, the moment bound $B=2$ does not depend on $d$, and $q$ may be chosen arbitrarily large.
	As a consequence, Theorem \ref{thm:eldan-sequential} yields a vanishing upper bound as long as $d\ll n^{1-\alpha}$ for some small $\alpha>0$.
\end{example}

To put Theorem \ref{thm:eldan-sequential} into perspective, one might consider the following coupling inequality, which is a handy special case of \cite[Cor.\ 3]{zaitsev2007}.

\begin{theorem}\label{thm:zaitsev}
	Let $X_1,\ldots, X_n$ be independent, $d$-dimensional random vectors with $\E \xi_t=0$, $t=1,\ldots,n$. 
	For some $q\geq 2$, let $L_q = \sum_{t=1}^n \E \|X_t\|^q<\infty$.
	Suppose that there exists $c <\infty$ such that
	\begin{align}
		\frac{ \sup_{\|v\|_2=1} v^T \Sigma_t v }{\inf_{\|v\|_2=1} v^T \Sigma_t v } = \frac{\lambda_{\max}(\Sigma_t)}{\lambda_{\min}(\Sigma_t)} \leq c,\qquad \Sigma_t = \Cov(X_t). \label{eqn:ass-matrix}
	\end{align}
	Then for any $\alpha>0$, one can construct, on a potentially different probability space, independent random vectors $\tilde{X}_1,\ldots, \tilde{X}_n$, such that $\tilde{X}_t \deq X_t$, and independent Gaussian random vectors $Y_1,\ldots, Y_n$ such that $Y_t\sim \mathcal{N}(0, \Sigma_t)$, which satisfy for any $\eta<q$
	\begin{align*}
		\left(\E\max_{k=1,\ldots, n} \left\| \frac{1}{\sqrt{n}} \sum_{t=1}^k (\tilde{X}_t-Y_t) \right\|_2^\eta\right)^\frac{1}{\eta} \leq C d^{\frac{15}{2}+\alpha} L_q^\frac{1}{q} n^{-\frac{1}{2}},
	\end{align*}
	for some universal constant $C=C(c,q,\eta,\alpha)$.
\end{theorem}

Without additional sparsity assumptions, the term $L_q^\frac{1}{q}$ is typically of order $\sqrt{d} n^\frac{1}{q}$.
Hence, the coupling of Theorem \ref{thm:zaitsev} has an approximation error of order $d^{8+\alpha} n^{\frac{1}{q}-\frac{1}{2}}$.
In contrast to Theorem \ref{thm:eldan-sequential}, this rate is optimal for fixed $d$, but far worse for high $d$.
In particular, it only tends to zero if at least $d\ll n^\frac{1}{16}$, considering $q\to\infty$. 
Moreover, condition \eqref{eqn:ass-matrix}, and also the slightly more general condition (1.13) of \cite{zaitsev2007}, are a severe constraint in high dimensions.
On the other hand, our result does not impose any restrictions on the covariance matrices.
Of course, an ideal result would combine the optimal rate in $n$ for fixed dimension, with a good dependence on the dimension $d$, and future work might explore different methods to extend the high-dimensional Wasserstein CLT to a sequential approximation.

\section{Strong approximation for dependent random vectors}\label{sec:dependent}

The sequential approximation result for independent random vectors may also be extended to nonstationary time series.
In this section, we consider a multivariate time series $X_{t}$, $t=1,\ldots, n$, taking values in $\R^{d}$.
To introduce our model framework, let $\epsilon_i, \tilde{\epsilon}_i$, $i\in\Z$, be two iid sequences of $U[0,1]$ random variables, mutually independent.
For any $t$, denote 
\begin{align*}
	\beps_t &= (\epsilon_t, \epsilon_{t-1},\ldots) \in\R^\infty, \\
	\tilde{\beps}_{t,j} &= (\epsilon_t,\ldots, \epsilon_{j+1}, \tilde{\epsilon}_j,  \epsilon_{j-1},\ldots) \in\R^\infty, \\
	\bar{\beps}_{t,j} &= (\epsilon_t,\ldots, \epsilon_{j+1}, \tilde{\epsilon}_j,  \tilde{\epsilon}_{j-1},\ldots) \in\R^\infty. 
\end{align*}
We assume that $X_{t} = G_{t}(\beps_t)$, for measurable mappings $G_{t}:\R^\infty\to\R^{d}$, $t=1,\ldots, n$, where we endow $\R^\infty$ with the $\sigma$-algebra generated by all finite projections.

To describe the ergodic properties of the time series $X_t$, we employ the physical dependence measure of \citet{Wu2005}. 
For $r,q\geq 2$, we define the physical dependence measure $\theta_{t,j,q,r}$ via
\begin{align}
	\theta_{t,j,q,r} = \left(\E \|G_{t}(\beps_t) - G_{t}(\tilde{\beps}_{t,t-j})\|_r^q\right)^\frac{1}{q} ,\qquad t=1,\ldots, n,\quad j\geq 0. \label{eqn:physical-theta} 
\end{align}
For our Gaussian approximation result, we suppose that the $\theta_{t,j,q,r}$ decay sufficiently fast as $j\to\infty$.
In particular, we assume that there exist $\Theta,\beta>0$, and a power $q\geq 2$, such that for all $t$, it holds
\begin{align}
	\begin{split}
	\theta_{t,j,q,2} &\leq \Theta \cdot (j\vee 1)^{-\beta},\qquad j\geq 0,\\
	\left(\E \|G_{t}(\beps_0)\|^q\right)^\frac{1}{q} &\leq \Theta. 
	\end{split}\label{eqn:ass-ergodic} \tag{G.1}
\end{align}
This assumption also implies that $(\E \|G_t(\beps_h) - G_t(\bar{\beps}_{h,h-j})\|^q)^\frac{1}{q} \leq \Theta \sum_{k=j}^\infty k^{-\beta}$, see e.g.\ \cite[C.1]{mies2021}.
Note that the dimension does not appear explicitly in this assumptions, but is implicitly contained in the factor $\Theta$. 
For example, if all $d$ components are iid, then one would expect $\Theta \approx \sqrt{d}$.

Moreover, we impose some regularity of the mapping $t\mapsto G_t(\beps_0)$.
Instead of classical smoothness assumptions, we formulate our conditions in terms of the total variation of the kernel.
In particular, assume that for some $\Gamma\geq 1$,
\begin{align}
	\sum_{t=2}^{n}\left( \E \|G_{t}(\beps_0) - G_{t-1}(\beps_0)\|^2 \right)^\frac{1}{2} \leq \Gamma\cdot\Theta. \label{eqn:ass-BV} \tag{G.2}
\end{align}
The factors $\Gamma$ and $\Theta$ will appear explicitly in the sequel. 
Also, we may replace $\beps_0$ by any $\beps_t$ in \eqref{eqn:ass-BV} since $\beps_0\sim\beps_t$. 
Note furthermore that \eqref{eqn:ass-BV} only uses the second moments.

In our approximation result below, Theorem \ref{thm:Gauss-ts}, we make the effect of $d$, $\Theta$, and $\Gamma$, fully explicit.
Hence, our theory also allows the values of $\Theta=\Theta_n$ and $\Gamma=\Gamma_n$ to depend on $n$, as well as the dimension $d=d_n$.  
This is illustrated by the following examples.

\begin{example}[Locally stationary process in low dimensions]
	It is instructive to consider Assumption \eqref{eqn:ass-BV} for the case of locally stationary time series in fixed dimension, say $d=1$.
	In our framework, a locally stationary time series is an array $X_t=X_{t,n}$ such that $X_{t,n} = G_{\frac{t}{n}}(\beps_t)$, where the kernel $G_u$ is defined for all $u\in[0,1]$.
	The concept of local stationarity has originally been introduced by \cite{Dahlhaus1997}, see also \cite{Dahlhaus2017}.
	The formulation of local stationarity in terms of a kernel $G_u$ has been introduced by \cite{Zhou2009}.
	If the mapping $u\mapsto G_u\in L_2(P)$ is of bounded variation, then condition \eqref{eqn:ass-BV} is satisfied with $\Theta\geq 1$ and $\Gamma = \max(1, \|G\|_{TV,[0,1]})$, where
	\begin{align*}
		\|G\|_{TV, [0,1]} = \sup_{\substack{0=u_0 < u_1 < \ldots < u_m=1 \\ m\in\N}} \sum_{k=1}^m \|G_{u_k} - G_{u_{k-1}}\|_{L_2(P)}.
	\end{align*}
	
	Local stationarity has been proposed in order to perform asymptotic inference in a non-stationary setting, where the kernel $G_u$ allows for a well-defined limit.
	In comparison to classical locally-stationary processes, strong approximation results of the kind presented in this paper enable us to consider models without a proper limit.
	For example, our framework also allows for the array $X_{t,n} = G_{\frac{t}{n^\phi}}(\beps_t)$ for some $\phi\in(0,1)$, if $G_u$ is defined for all $u\in[0,\infty)$.
	Then $\Gamma$ in \eqref{eqn:ass-BV} corresponds to the total variation of $G_u$ on the interval $[0, n^{1-\phi}]$, i.e.\ $\Gamma= \|G\|_{TV, [0,n^{1-\phi}]}$ which might grow with $n$. 
	If $u\mapsto G_u\in L_q(P)$ is Lipschitz continuous with Lipschitz-constant $L$, then $\|G\|_{TV, [0,n^{1-\phi}]} \leq L\, n^{1-\phi}$, and a corresponding factor $\Gamma$ will scale as $\Gamma_n\asymp n^{1-\phi}$.
	Our results cover these cases by making the effect of $\Gamma$ explicit.
\end{example}

\begin{example}[Locally stationary process in high dimensions]
	The locally stationary process may also be defined in increasing dimension $d=d_n$ as $X_{t,n} = G_{n,\frac{t}{n^\phi}}(\beps_t)$.
	Here, for each $n\in\N$ and for each $u\in[0,\infty)$, $G_{n,u}:\R^\infty\to\R^{d_n}$, and the individual components $[X_{t,n}]_{i} = G_{n,\frac{t}{n^\phi}}^i (\beps_t)$, $i=1,\ldots, d_n$, are univariate locally stationary processes.
	To ensure conditions \eqref{eqn:ass-ergodic} and \eqref{eqn:ass-BV}, it is convenient to impose assumptions on the marginal kernels $G^i_{n,u}$ only.
	The marginal analogy of \eqref{eqn:ass-ergodic} is to require that, for all $n\in\N$, $u\in[0,\infty)$, and all coordinates $i=1,\ldots, d_n$,
	\begin{align*}
		\begin{split}
			\left(\E \|G_{n,u}^i(\beps_t) - G_{n,u}^i(\tilde{\beps}_{t,t-j})\|_r^q\right)^\frac{1}{q} &\leq \bar{\Theta} \cdot (j\vee 1)^{-\beta},\qquad j\geq 0,\\
			\left(\E \|G_{n,u}^i(\beps_0)\|^q\right)^\frac{1}{q} &\leq \bar{\Theta}, 
		\end{split}\label{eqn:ass-ergodic-LS} \tag{G.1*}
	\end{align*}
	for some $\bar{\Theta}$ not depending on $n$.
	Regarding the regularity, we may impose that the mappings $u\mapsto G_{n,u}^i\in L_q(P)$ are Lipschitz continuous, i.e.\ for all $n\in\N$, $i=1,\ldots, d_n$,
	\begin{align}
		\|G_{n,u}^i - G_{n,v}^i\|_{L_q(P)} \leq L \bar{\Theta} |u-v|,\qquad u,v \in[0,\infty). \label{eqn:ass-lipschitz-LS} \tag{G.2*}
	\end{align}
	It is then straightforward to check that the array $X_{t,n}$ satisfies conditions \eqref{eqn:ass-ergodic} and \eqref{eqn:ass-BV}, with $\Theta=\Theta_n=\sqrt{d}_n \bar{\Theta}$, and $\Gamma=\Gamma_n = \max(L n^{1-\phi},1)$.
	This example highlights the intuition about the two factors $\Gamma$ and $\Theta$. 
	While $\Theta$ typically contains the effect of the dimension, the size of $\Gamma$ reflects the regularity in time, i.e.\ the degree of nonstationarity.
\end{example}

The main result of this section is the following theorem about Gaussian approximation of the time series $X_t$, making the dependence on $d, \Theta$, and $\Gamma$, explicit.
To formulate our result, define the two rates
\begin{align*}
	\chi(q,\beta) = \begin{cases}
		\frac{q-2}{6q-4}, & \beta \geq \frac{3}{2}, \\
		\frac{(\beta-1)(q-2)}{q(4\beta-3)-2}, &\beta\in(1,\frac{3}{2}),
	\end{cases}
	\quad
	\xi(q,\beta) = \begin{cases}
		\frac{q-2}{6q-4}, & \beta \geq 3, \\
		\frac{(\beta-2)(q-2)}{(4\beta-6)q-4},& \frac{3+\frac{2}{q}}{1+\frac{2}{q}} < \beta < 3, \\
		\frac{1}{2}-\frac{1}{\beta}, & 2< \beta \leq \frac{3+\frac{2}{q}}{1+\frac{2}{q}}.\\
	\end{cases}
\end{align*}

\begin{theorem}\label{thm:Gauss-ts}
	Let $X_t=G_t(\beps_t)$ with $\E (X_t) = 0$ satisfy \eqref{eqn:ass-ergodic}, for some $q>2$, and $\beta>1$, and suppose $d\leq c n$ for some $c>0$.
	Then, on a potentially different probability space, there exist random vectors $(X_t')_{t=1}^n\deq (X_t)_{t=1}^n$ and independent, mean zero, Gaussian random vectors $Y_t'$ such that
	\begin{align}
		\left(\E \max_{k\leq n} \left\|\frac{1}{\sqrt{n}}\sum_{t=1}^k  (X_t' - Y_t') \right\|^2 \right)^\frac{1}{2}
		&\leq C \Theta \sqrt{\log(n)} 
		   \left( \frac{d}{n} \right)^{\chi(q,\beta)} . \label{eqn:Gauss-ts-1}
	\end{align} 
	for some universal constant $C$ depending on $q$, $c$, and $\beta$. 
	
	If $\beta>2$, the local long-run covariance matrix $\Sigma_t = \sum_{h=-\infty}^\infty \Cov(G_t(\beps_0),G_t(\beps_h))$ is well-defined.
	If \eqref{eqn:ass-BV} is satisfied aswell, then there exist random vectors $(X_t')_{t=1}^n\deq (X_t)_{t=1}^n$ and independent, mean zero, Gaussian random vectors $Y_t^*\sim\mathcal{N}(0,\Sigma_t)$ such that
	\begin{align}
		\left(\E \max_{k\leq n} \left\|\frac{1}{\sqrt{n}}\sum_{t=1}^k  (X_t' - Y_t^*)   \right\|^2 \right)^\frac{1}{2} 
		&\leq C\Theta  \Gamma^{\frac{1}{2} \frac{\beta-2}{\beta-1}} \sqrt{\log(n)} \left( \frac{d}{n} \right)^{\xi(q,\beta)}. 
		\label{eqn:Gauss-ts-2}
	\end{align}
\end{theorem}

\begin{remark}
	In the ideal situation where $q\to\infty$ and $\beta\to\infty$, the rate of \eqref{eqn:Gauss-ts-1} approaches the order $\mathcal{O}(\Theta \sqrt{\log(n)} (d/n)^\frac{1}{6})$.
	For the typical scaling $\Theta \asymp \sqrt{d}$, this rate becomes $\mathcal{O}(\sqrt{\log(n)} (d^4/n)^\frac{1}{6})$.
	Hence, the bound is non-trivial if $d=o(n^{\frac{1}{4}-\alpha})$ for some $\alpha>0$, which is slightly worse than the restriction $d=o(n^{\frac{1}{3}-\alpha})$ for the independent case, and the restriction $d=o(n^{\frac{1}{2}-\alpha})$ for the non-sequential result of \cite{eldan2020}.
\end{remark}

\begin{remark}
	Compared to the multivariate Gaussian approximation result of \cite{Karmakar2020}, we allow for increasing dimension $d$, and we do not need to impose a lower bound on the eigenvalues of the long run covariance matrix.
	In particular, condition (2.B) therein is similar to condition \eqref{eqn:ass-matrix} in the coupling of \cite{zaitsev2007}, and may be too restrictive in a high-dimensional setting.
	On the other hand, the rate of Theorem \ref{thm:Gauss-ts} is suboptimal with respect to $n$.
\end{remark}

The rate \eqref{eqn:Gauss-ts-1} is better and requires fewer assumptions than the rate of approximation \eqref{eqn:Gauss-ts-2}.
However, the covariance of the Gaussian random vectors $Y_t'$, which is given by formula \eqref{eqn:Y-ts-cov} in the proof, is not very handy. 
In contrast, the second approximation via the $Y_t^*$ gives a manageable expression for the approximating Gaussian process, and hence lends itself to statistical inference.
In Section \ref{sec:bootstrap}, we discuss how to perform feasible inference based on the approximation \eqref{eqn:Gauss-ts-2}.

Since our model contains locally stationary processes as a special case, Theorem \ref{thm:Gauss-ts} may also be regarded as an extension of the limit theory for these processes.
For example, \cite{Zhou2013} derives a functional central limit theorem under the so-called piecewise locally stationarity (PLS) assumption.
In \cite{mies2021}, we extended the result of Zhou for a more general regularity condition, similar to \eqref{eqn:ass-BV}.
While these functional central limit theorems allow for a convenient description of the convergence and the limit distribution, Theorem \ref{thm:Gauss-ts} quantifies the distributional convergence via explicit rates, and it is not restricted to the locally stationary model.

The proof of Theorem \ref{thm:Gauss-ts} proceeds by splitting the random vectors $X_t$ into blocks of consecutive terms which decouple by virtue of Assumption \eqref{eqn:ass-ergodic}, such that the Gaussian approximation results for independent random vectors, i.e.\ Theorem \ref{thm:eldan-sequential}, may be applied.
A second major ingredient in the proof of our main result is a bound on the moments of sums of dependent random vectors.
For the case of univariate stationary time series, \citet[Theorem 1]{Liu2013} present a Rosenthal-type inequality which depends explicitly on the functional dependence measure of the time series $X_t$.
Theorem \ref{thm:liu} below extends the result of Liu et al.\ to the high-dimensional, non-stationary setting, and might be of independent interest.

\begin{theorem}\label{thm:liu}
	Let $X_{t} = G_{t}(\beps_t) \in L_q(P)$, $t=1,\ldots, n$, with $\theta_{t,j,q,r}$ as in \eqref{eqn:physical-theta} for some $2\leq r \leq q<\infty$.
	There exists a universal constant $C=C(q,r)$, such that for all $n\in\N$,
	\begin{align}
		\begin{split}
			\left(\E\,  \max_{k\leq n} \left\|\sum_{t=1}^k ( X_{t} - \E X_{t}) \right\|_r^q\right)^{\frac{1}{q}}
			&\leq C n^{\frac{1}{2}-\frac{1}{q}} \sum_{j=1}^\infty \left( \sum_{t=1}^n \theta_{t,j,q,r}^q\right)^\frac{1}{q}  \\
			&\leq C n^\frac{1}{2}\sum_{j=1}^\infty \max_{t\leq n}\theta_{t,j,q,r} . 
		\end{split}\label{eqn:liu-1}
	\end{align}
	In the special case $r=2$, the inequality may be improved to
	\begin{align}
		&\quad \left(\E\,  \max_{k\leq n} \left\|\sum_{t=1}^k ( X_{t} - \E X_{t}) \right\|_2^q\right)^{\frac{1}{q}}\nonumber \\
		&\leq C \sum_{j=1}^\infty  (j\wedge n)^{\frac{1}{2}-\frac{1}{q}} \left( \sum_{t=1}^n \theta_{t,j,q,2}^q \right)^\frac{1}{q} + C \sum_{j=1}^n \left( \sum_{t=1}^n \theta_{t,j,2,2}^2 \right)^\frac{1}{2}. \label{eqn:liu-2}
	\end{align}
\end{theorem}

While the choice of $r$ is irrelevant in finite dimensions, it might have a great impact if $d$ is large.
In particular, the bound in Theorem \ref{thm:liu} is dimension-agnostic, while the effect of the dimension is implicitly contained in the $\theta_{t,j,q,r}$.

\section{Feasible Gaussian approximation}\label{sec:bootstrap}

To employ Theorem \ref{thm:Gauss-ts} for statistical inference, we need to find a feasible approximation of the limiting Gaussian process. 
Since the covariance structure of the approximating Gaussian process is given explicitly in our result \eqref{eqn:Gauss-ts-2} by the local long run covariance matrices $\Sigma_t$, it suffices to find a suitable estimator of this covariance.
Indeed, Proposition \ref{prop:var-consistency} below reveals that we do not need to estimate each $\Sigma_t$ individually, but we only need an estimator $\hat{Q}(k)$ of the cumulative covariance process $Q(k) = \sum_{t=1}^k \Sigma_t$ such that $\max_k \|\hat{Q}(k)-Q(k)\|_{\tr}$ is small.
Denoting the outer product as $v^{\otimes 2} = v v^T$ for $v\in\R^d$, we suggest the estimator 
\begin{align*}
	\hat{Q}(k) &= \sum_{t=b}^k \frac{1}{b} \left( \sum_{s=t-b+1}^t X_s \right)^{\otimes 2}.
\end{align*}
The same estimator has been suggested in the univariate, nonstationary setting by \cite{Zhou2013}.
In the univariate, stationary setting, the estimator $\hat{Q}(n)$ has been previously studied by \cite{Peligrad1995}, and with non-overlapping blocks by \cite{Carlstein1986}.

The estimation error of $\hat{Q}(k)$ may be bounded as follows.

\begin{theorem}\label{thm:cov-est}
	Let $X_t=G_t(\beps_t)$ satisfy \eqref{eqn:ass-ergodic} with $q\geq 4$ and $\beta>2$, and \eqref{eqn:ass-BV}.
	Then
	\begin{align*}
		\E \max_{k=1,\ldots,n}\left\|  \hat{Q}(k) - \sum_{t=1}^k \Sigma_t\right\|_\tr
		&\leq C  \Theta^2 \left( \Gamma \sqrt{b} + \sqrt{n d b} + n b^{-1} + n b^{2-\beta} \right)
	\end{align*}
	for a universal factor $C$ depending on $\beta$ and $q$ only.
\end{theorem}

The first error term is a bias term due to the nonstationarity of the process. 
The second term corresponds to the statistical error of estimation. 
The third and fourth term are bias terms because only finitely many lags are considered.
We note that terms two and three are also present in the result of \cite[Theorem 4]{Zhou2013}, while the first term is negligible, and the fourth term appears because, unlike Zhou, we do not assume a geometric decay of the dependence measure.

In order to perform statistical inference based on the estimator $\hat{Q}(k)$, the following proposition is of central importance.

\begin{proposition}\label{prop:var-consistency}
	Let $\Sigma_t,\Sigma_t'\in\R^{d\times d}$ be symmetric, positive definite matrices, for $t=1,\ldots,n$, and consider independent random vectors $Y_t\sim\mathcal{N}(0,\Sigma_t)$.
	On a potentially larger probability space, there exist independent random vectors $Y_t'\sim\mathcal{N}(0,\Sigma_t')$, such that
	\begin{align*}
		\E \max_{k=1,\ldots,n}\left\| \sum_{t=1}^k Y_t - \sum_{t=1}^k Y_t' \right\|^2 
			&\leq C \log(n) \left[\sqrt{n\delta\rho} + \rho \right], \\
		\text{where}\quad\delta = \max_{k=1,\ldots,n} \left\| \sum_{t=1}^k \Sigma_t  -  \sum_{t=1}^k\Sigma_t' \right\|_{\tr},
			&\qquad \rho = \max_{t=1,\ldots, n} \|\Sigma_t\|_{\tr}.
	\end{align*}
\end{proposition}

Let us briefly discuss the implications of Proposition \ref{prop:var-consistency} by describing two statistical applications.
The first example is a sequential test for the mean value.
For a multivariate time series $X_{t}=X_{t,n}$, we want to test whether $\E(X_t)=0$ for all $t=1,\ldots, n$, i.e.\ we consider the hypothesis
\begin{align*}
	H_0^* : \E(X_t) = 0 \text{ for all } t=1,\ldots,n.
\end{align*} 
A suitable test statistic is given by $T_n^*(X_1,\ldots, X_n) = \max_{k=1,\ldots,n} \| \frac{1}{\sqrt{n}} \sum_{t=1}^k X_t \|$, and we reject $H_0^*$ for large values of $T_n^*$.

The second example is a change-point test for high-dimensional time series. 
Here, we consider the null-hypothesis
\begin{align*}
	H_0^\diamond: \E(X_t) = \E(X_1) \text{ for all } t=1,\ldots, n.
\end{align*}
For this testing problem, one may employ the CUSUM statistic 
\begin{align*}
	T_n^\diamond(X_1,\ldots, X_n) = \max_{k=1,\ldots, n} \frac{1}{\sqrt{n}}\left\| \sum_{t=1}^k X_t - \frac{k}{n} \sum_{t=1}^n X_t  \right\|,
\end{align*} 
and we reject $H_0^\diamond$ for large values of $T_n^\diamond$.

Both statistics are compatible with the strong Gaussian approximation.
To be precise, for $T_n \in \{T_n^*, T_n^\diamond\}$, it holds that
\begin{align}
	\left| T_n(X_1,\ldots, X_n) - T_n(Y_1,\ldots, Y_n) \right| \leq \frac{1}{\sqrt{n}} \max_{k=1,\ldots,n}\left\| \sum_{t=1}^k X_t - \sum_{t=1}^k Y_t \right\|. \label{eqn:Tn-continuity}
\end{align}
Hence, we may use our approximation results to derive an asymptotically exact test for the hypotheses $H_0^*$ and $H_0^\diamond$.
To this end, we determine a critical value as follows.
For any increasing process $Q_t$ of symmetric, positive semidefnite matrices, let $\tilde{\Sigma}_t = Q(t) - Q(t-1)$.
Then $\tilde{\Sigma}_t$ are symmetric positive semidefinite matrices.
Let $Z_t \sim \mathcal{N}(0,\tilde{\Sigma}_t)$ be independent Gaussian random vectors.
For any significance level $\alpha\in(0,1)$, and for any realization, we may then find a suitable quantile $a_\alpha=a_\alpha(Q)$ such that
\begin{align*}
	a_\alpha(Q) = \inf\left\{ a\,:\, P\left(T_n(Z_1,\ldots, Z_n)>a \right) \leq \alpha \right\}.
\end{align*}
We reject the the null hypothesis if 
\begin{align}
	T_n(X_1,\ldots, X_n) > a_{\alpha-\nu}(\hat{Q}) + \tau. \label{eqn:sequential-test}
\end{align}
For any realization $\hat{Q}$, the quantile $a(\hat{Q})$ may be determined numerically via Monte Carlo simulations, which may be regarded as a type of bootstrap inference.
To account for the estimation error of $\hat{Q}$ and for the error of the Gaussian approximation, we add an offset $\tau=\tau_n$ such that $\tau_n\to 0$ as $n\to\infty$, and $\nu=\nu_n\to 0$.

\begin{proposition}\label{prop:seq-test}
	Let $X_t=X_{t,n} = G_{t,n}(\beps_t)$, $t=1,\ldots, n$, be an array of $d_n$-variate time series, such that each kernel $G_{t,n}$ satisfies \eqref{eqn:ass-ergodic} and \eqref{eqn:ass-BV} for some $q>4, \beta>2$, and with factors $\Theta_n$ and $\Gamma_n$. 
	
	If $\tau_n\to 0$ and $\nu_n\to 0$ are chosen such that 
	\begin{align}
		\tau_n &\gg \sqrt{\log(n)} \Theta_n \left\{ (\tfrac{d_n}{n})^{\xi(q,\beta)} + \nu_n^{-\frac{1}{2}}\left( \Gamma_n^{\frac{1}{4}} n^{-\frac{1}{4}} b_n^{\frac{1}{8}} + n^{-\frac{1}{8}} d_n^\frac{1}{8} b_n^{\frac{1}{8}} +  b_n^{-\frac{1}{4}} +  b_n^{\frac{2-\beta}{4}} +n^{-\frac{1}{2}} \right) \right\}, \label{eqn:seq-test-ass}
	\end{align}	
	then for any statistic $T_n$ satisfying \eqref{eqn:Tn-continuity},
	\begin{align*}
		\limsup_{n\to\infty} P\left( T_n(X_1,\ldots, X_n) > a_{\alpha-\nu_n}(\hat{Q}) + \tau_n \right) \leq \alpha
	\end{align*}
\end{proposition}

Condition \eqref{eqn:seq-test-ass} looks rather complicated since it combines all constraints on the dimension $d_n$, the bounds $\Theta_n$ and $\Gamma_n$, the window size $b_n$, and the offsets $\nu_n, \tau_n$.
It is instructive to consider two special cases.
First, if $\Gamma_n, \Theta_n$, and $d_n$ are all constant, then a conservative choice satisfying \eqref{eqn:seq-test-ass} is $\nu_n = \tau_n = 1/ \log(n)$, and $b_n \asymp n^\zeta$ for some $\zeta \in (0, \frac{1}{2})$. 
Secondly, in the high-dimensional case $d_n\to\infty$ and $\Theta_n \asymp \sqrt{d_n}$, we also need to require that $d_n^{1+\frac{1}{2\xi(q,\beta)}}/n = \mathcal{O}(n^{-\delta})$ for some small $\delta>0$ so that the error of the Gaussian approximation is negligible. 
In the limiting case where $\beta\geq 3$ and $q\to\infty$, this corresponds to $d_n = \mathcal{O}(n^{\frac{1}{4}-\delta})$ for some $\delta>0$.
Note that this restriction on the dimension also implies that for all choices $b_n=n^\zeta, \zeta\in(0,\frac{1}{2})$, condition \eqref{eqn:seq-test-ass} is satisfied.

Proposition \ref{prop:seq-test} shows that the suggested test based on our strong approximation result asymptotically maintains a specified type-I error of at most $\alpha\in(0,1)$, i.e., the test is conservative. 
To obtain the exact size of the test, we would need to exploit some regularity of the mapping $\alpha\mapsto a_\alpha$. 
However, these quantiles also depend on the specific statistic $T_n$, thus requiring a case-by-case analysis. 
This question is out of scope of the present article.

\section{Proofs}\label{sec:proofs}

\subsection{Preliminaries}

\begin{lemma}\label{lem:gauss-norm}
	For any $q\geq 2$, there exists some $C_q$ such that for all $d\in\N$ and any centered $d$-variate Gaussian random vector $Y$, $(\E \|Y\|^q)^\frac{1}{q} \leq C_q (\E \|Y\|^2)^\frac{1}{2}$.
\end{lemma}
\begin{proof}[Proof of Lemma \ref{lem:gauss-norm}]
	Let $Y\sim\mathcal{N}(0,\Sigma)$. 
	Then $\|Y\|^2 = \sum_{j=1}^d \sigma_j^2 \delta_j^2$, where $\sigma_j^2$ are the eigenvalues of $\Sigma$, and $\delta_j$ are iid standard normal random variables.
	Hence, via the Minkowski inequality,
	\begin{align*}
		(\E\|Y\|^q )^\frac{1}{q} 
		= \left(\E \left(\sum_{j=1}^d \sigma_j^2 \delta_j^2\right)^\frac{q}{2} \right)^{\frac{2}{q}\frac{1}{2}} 
		\leq \left( \sum_{j=1}^d \sigma_j^2 (\E |\delta_j|^q)^\frac{2}{q} \right)^\frac{1}{2}
		\leq C_q \left( \sum_{j=1}^d \sigma_j^2 \right)^\frac{1}{2}.
	\end{align*}
\end{proof}

\begin{lemma}[2-Wasserstein bound]\label{lem:Gauss-cov}
	Let $\Sigma_1,\Sigma_2\in\R^{d\times d}$ be symmetric, positive semidefinite matrices, and let $v_i\in\R^d, \lambda_i\in\R$, $i=1,\ldots, d$, be the eigenvectors and eigenvalues of $\Delta = \Sigma_2-\Sigma_1$.
	Define the matrix $|\Delta| = \sum_{i=1}^d v_i v_i^T |\lambda_i|$.
	Consider a random vector $Y\sim\mathcal{N}(0,\Sigma_1)$ defined on a sufficiently rich probability space.
	Then there exists a random vector $\eta\sim\mathcal{N}(0,|\Delta|)$ such that $Y+\eta\sim \mathcal{N}(0,\Sigma_2)$.
\end{lemma}
\begin{proof}[Proof of Lemma \ref{lem:Gauss-cov}]
	Introduce the symmetric positive semidefinite matrices
	\begin{align*}
		\Delta_+ = \sum_{i=1}^d v_i v_i^T (\lambda_i \vee 0), \qquad \Delta_- = \sum_{i=1}^d v_i v_i^T (-\lambda_i \vee 0),
	\end{align*}
	such that $|\Delta| = \Delta_+ + \Delta_-$, and $\Delta=\Delta_+-\Delta_-$.
	Moreover, it holds that $\Delta_- \prec \Sigma_1$, where $\prec$ denotes the partial ordering of positive semidefinite matrices, because
	\begin{align*}
		v_i^T \Delta_- v_i = -\lambda_i = v_i^T(\Sigma_1-\Sigma_2) v_i &\leq v_i^T \Sigma_1v_i,\qquad \lambda_i\leq 0,\\
		v_i^T \Delta_- v_i = 0 &\leq v_i^T \Sigma_1 v_i,\qquad \lambda_i >0.
	\end{align*}
	Define the matrix $\Sigma\in\R^{3d\times 3d}$ by
	\begin{align*}
		\Sigma = \begin{pmatrix}
			\Sigma_1 & -\Delta_- & 0 \\
			-\Delta_-^T & \Delta_- & 0 \\
			0 & 0 & \Delta_+
		\end{pmatrix}.
	\end{align*}
	Because $\Delta_- \prec \Sigma_1$, the symmetric matrix $\Sigma$ is positive semidefinite.
	Hence, there exist $d$-dimensional random vectors $\eta_1,\eta_2$, such that $(Y,\eta_1,\eta_2)\sim\mathcal{N}(0,\Sigma)$.
	By our choice of $\Sigma$, we have
	\begin{align*}
		\Cov(Y+\eta_1+\eta_2) &= \Sigma_1 - \Delta_- + \Delta_+ = \Sigma_2, \\
		\Cov(\eta_1+\eta_2) &= \Delta_+ + \Delta_- = |\Delta|.
	\end{align*}
	Hence, we established the claim of the Lemma with $\eta = \eta_1 + \eta_2$.
\end{proof}

\begin{lemma}\label{lem:trace-bound}
	For two vectors $x,y\in\R^d$, it holds $\|x y^T\|_\tr =\|x y^T\|_F = \|x\|_2 \|y\|_2$.
\end{lemma}
\begin{proof}[Proof of Lemma \ref{lem:trace-bound}]
	Observe that $A= (xy^T)(xy^T)^T = xy^T y x^T = \|y\|_2^2 x x^T$, such that the eigenvalues of $A$ are $(\|y\|_2^2 \|x\|_2^2, 0,\ldots, 0)$. 
	Hence, $\|xy^T\|_\tr = \tr(\sqrt{A}) = \|y\|_2 \|x\|_2 = \|xy^T\|_F$. 
\end{proof}

\begin{proposition}\label{prop:cov}
	Let $G_{t}:\R^\infty\to\R^{d}$ satisfy \eqref{eqn:ass-ergodic} for $q\geq 2$.
	Denote 
	\begin{align*}
		\gamma_{t}(h) = \Cov \left[ G_{t}(\beps_0), G_{t}(\beps_h) \right] \in\R^{d\times d}.
	\end{align*}
	Then $\|\gamma_{t}(h)\|_\tr \leq \Theta^2 \sum_{j=h}^\infty j^{-\beta}$, where $\|\cdot \|_\tr$ denotes the trace norm.
	Hence, if $\beta>2$, then the long-run covariance $\gamma_t = \sum_{h=-\infty}^\infty \gamma_t(h)$ is well-defined.
\end{proposition}
\begin{proof}[Proof of Proposition \ref{prop:cov}]
	We have
	\begin{align*}
		\gamma_{t}(h) 
		&=  \Cov\left[ G_{t}(\beps_0), G_{t}(\beps_h) \right] \\
		&= \Cov\left[ G_{t}(\beps_0), G_{t}(\bar{\beps}_{h,0}) \right] + \Cov\left[  G_{t}(\beps_0), G_{t}(\beps_h) - G_{t}(\bar{\beps}_{h,0}) \right] \\
		&= \Cov\left[  G_{t}(\beps_0), G_{t}(\beps_h) - G_{t}(\bar{\beps}_{h,0}) \right]
	\end{align*}
	since $\beps_0$ and $\bar{\beps}_{h,0}$ are independent.
	Now, we use Lemma \ref{lem:trace-bound}, assumption \eqref{eqn:ass-ergodic}, and the triangle inequality for expectations, to obtain
	\begin{align*}
		\|\Cov\left[  G_{t}(\beps_0), G_{t}(\beps_h) - G_{t}(\bar{\beps}_{h,0}) \right]\|_\tr 
		&\leq \E \|G_{t}(\beps_0)\|_2 \|G_{t}(\beps_h) - G_{t}(\bar{\beps}_{h,0})\|_2  \\
		&\leq \sqrt{\E \|G_{t}(\beps_0)\|_2^2} \sqrt{ \E \|G_{t}(\beps_h) - G_{t}(\bar{\beps}_{h,0})\|_2^2 } \\
		&\leq \Theta\, \sum_{j=h}^\infty \Theta j^{-\beta}.
	\end{align*}
	Finiteness of the long-run covariance is an immediate consequence.
\end{proof}

\subsection{Gaussian approximation for independent random vectors}

The proof of Theorem \ref{thm:wasserstein-eldan} is largely analogous to the work of \cite{eldan2020}.
The difference is that we account for the unboundedness of the random vectors, and we allow the random vectors to be non-identically distributed. 
Among other arguments, we employ the following technical result, which is also used in the proof of \cite[Thm.\ 10]{eldan2020}.
Since the corresponding step in the work of Eldan et al.\ is rather brief, we present additional details. 

\begin{proposition}\label{prop:BM-change}
	Let $(\Omega,\F,(\F_u)_{u\geq 0},P)$ be a filtered probability space, and let $B_{u,1}$ and $B_{u,2}$ be two independent Brownian motions w.r.t.\ $\F_u$, and let $A_u, \bar{A}_u\in\R^{d\times d}$ be adapted processes.
	Then there exists a Brownian motion $\bar{B}_u$ such that for all $T\geq 0$,
	\begin{align*}
		\int_0^T A_u\, dB_{u,1} + \int_0^T \bar{A}_u\, dB_{u,2} 
		&= \int_0^T \sqrt{A_u A_u^T + \bar{A}_u \bar{A}_u^T} \, d\bar{B}_u.
	\end{align*}
\end{proposition}
\begin{proof}[Proof of Proposition \ref{prop:BM-change}]
	Introduce the matrix $\mathcal{A}_u = \sqrt{A_u A_u^T + \bar{A}_u \bar{A}_u^T}$.
	Since $\mathcal{A}_u$ is symmetric positive semidefinite, we may write $\mathcal{A}_u = \sum_{j=1}^m \lambda_j v_j v_j^T$ for orthogonal vectors $v_j$, eigenvalues $\lambda_j>0$, and some $0\leq m\leq d$.
	By extending $v_j$ to a full orthonormal basis, we define the regular matrix $\bar{\mathcal{A}}_u = \mathcal{A}_u + \sum_{j=m+1}^d v_j v_j^T$.
	Now let $B_{u,3}$ be a third Brownian motion, independent from $B_{u,1}$ and $B_{u,2}$, and define
	\begin{align*}
		\bar{B}_u &= \int_0^u \bar{\mathcal{A}}_z^{-1} A_z\,dB_{z,1} 
		+ \int_0^u \bar{\mathcal{A}}_z^{-1} \bar{A}_z\,dB_{z,2} 
		+ \int_{0}^u (\bar{\mathcal{A}}_z - \mathcal{A}_z)\, dB_{z,3}. 
	\end{align*}
	Note that $(\bar{\mathcal{A}}_u - \mathcal{A}_u) = \sum_{j=m+1}^d v_j v_j^T$ is a projection matrix.
	Then $\bar{B}_u$ is a continuous local martingale with quadratic variation 
	\begin{align*}
		[\bar{B}]_u &= \int_0^u \bar{\mathcal{A}}_z^{-1} A_u A_u^T \bar{\mathcal{A}}_u^{-1} \, du 
		+ \int_0^u \bar{\mathcal{A}}_z^{-1} \bar{A}_z \bar{A}_z^T \bar{\mathcal{A}}_z^{-1} \, dz
		+ \int_0^u (\bar{\mathcal{A}}_z - \mathcal{A}_z) \, dz \\
		&= \int_0^u I_d\, zs.
	\end{align*}
	Thus, $\bar{B}_u$ is a standard $d$-variate Brownian motion by virtue of Lévy's characterization, see \cite[Thm.\ 3.3.16]{Karatzas1998}.
	Moreover, $\mathcal{A}_u \bar{\mathcal{A}}_u^{-1} A_u = A_u$, and $\mathcal{A}_u \bar{\mathcal{A}}_u^{-1} \bar{A}_u = \bar{A}_u$, and $\mathcal{A}_u (\bar{\mathcal{A}}_u - \mathcal{A}_u) = 0$.
	Thus,
	\begin{align*}
		\int_0^T \mathcal{A}_u \, d\bar{B}_u 
		&= \int_0^T A_u \, dB_{u,1} + \int_0^T \bar{A}_u \, dB_{u,2},
	\end{align*}
	completing the proof.
\end{proof}

\begin{proof}[Proof of Theorem \ref{thm:wasserstein-eldan}]
	We start by truncating the random variables as $X_t^\diamond = \frac{X_t}{\|X_t\|} (\|X_t\| \wedge \beta_t)$, such that $\|X_t^\diamond\|\leq \beta_t$.
	Then
	\begin{align*}
		\E \|X_t - X_t^\diamond\|^2 
		&= \E [(\|X_t\| - \beta_t)^2\vee 0] \\
		&= \int_{0}^\infty P(\|X_t\|>\sqrt{x}+\beta_t)\, dx \\
		&=  \int_{0}^\infty 2 y P(\|X_t\|>y+\beta_t)\, dy \\
		&\leq  \int_{\beta_t}^\infty 2 z P(\|X_t\|>z)\, dz \\
		&\leq 2\int_{\beta_t}^\infty \frac{b_t^q}{z^{q-1}}\, dz 
		\quad = \frac{2}{q-2} b_t^q \beta_t^{2-q}.
	\end{align*}
	This also implies $\|\E(X_t^\diamond)\|^2 = \|\E(X_t - X_t^\diamond)\|^2 \leq \frac{2}{q-2} b_t^q \beta_t^{2-q}$. 
	Setting $\bar{X}_t = X_t^\diamond - \E(X_t^\diamond)$, we obtain $\E \|X_t - \bar{X}_t\|^2 \leq \frac{4}{q-2} b_t^q \beta_t^{2-q}$, and note that $\|\bar{X}_t\|\leq 2\beta_t$.
	The threshold values $\beta_t$ will be chosen later.
	
	We proceed to approximate $\frac{1}{\sqrt{n}}\sum_{t=1}^n \bar{X}_t$ by a Gaussian random vector.
	According to \cite{eldan2020}, there exist independent Brownian motions $(B^t_s)_{s\geq 0}$, $t=1,\ldots, n$, stopping times $\tau_t$ and adapted processes $\Gamma_s^t\in\R^{d\times d}$ such that $\Gamma^t_s=0$ for $s\geq \tau_t$, and $\bar{X}_t \deq \tilde{X}_t = \int_0^{\infty} \Gamma_s^t \, dB_s^t$.
	Moreover, each $\Gamma_s^t$ is a symmetric positive semidefinite projection matrix, i.e.\ $\Gamma_s^t \Gamma_s^t = \Gamma_{s}^t$.
	By applying Proposition \ref{prop:BM-change} inductively, we find another Brownian motion $B_s$ such that $\tilde{S}_n=\frac{1}{\sqrt{n}}\sum_{t=1}^n \tilde{X}_t = \int_0^\infty \tilde{\Gamma}_s d B_t$, where $\tilde{\Gamma}_s = \sqrt{ \frac{1}{n} \sum_{t=1}^n (\Gamma_s^t)^2 }$.
	Denote 
	\begin{align*}
		\bar{\Y}_n=\int_0^\infty \sqrt{\E (\tilde{\Gamma}_s)^2}\, dB_s \sim \mathcal{N}(0, \Cov(\tilde{S}_n)).
	\end{align*}
	Analogously to the proof of Eldan et al., we find that
	\begin{align*}
		\E \|\tilde{S}_n - \bar{\Y}_n\|^2 
		&\leq \int_0^\infty \E \, \tr\left[  \left(\tilde{\Gamma}_s - \sqrt{ \frac{1}{n} \sum_{t=1}^n \E (\Gamma_s^t)^2 } \right)^2  \right]\, ds,
	\end{align*}
	and the integrand may be bounded as
	\begin{align*}
		\E \, \tr\left[  \left(\tilde{\Gamma}_s - \sqrt{ \frac{1}{n} \sum_{t=1}^n \E (\Gamma_s^t)^2 } \right)^2  \right]
		&\leq 4 \frac{1}{n} \sum_{t=1}^n \E \,\tr[(\Gamma_{s}^t)^2] 
		= 4 \frac{1}{n} \sum_{t=1}^n \E \,\tr[\Gamma_{s}^t], \\
		\E \, \tr\left[  \left(\tilde{\Gamma}_t - \sqrt{ \frac{1}{n} \sum_{t=1}^n \E (\Gamma_s^t)^2 } \right)^2  \right]
		&\leq \frac{1}{n} \tr \left[ \E \left( \frac{1}{n} \sum_{t=1}^n (\Gamma_s^t)^4 \right) \E \left( \frac{1}{n} \sum_{t=1}^n (\Gamma_s^t)^2 \right)^\dagger \right] \\
		&\leq \frac{d}{n},
	\end{align*} 
	where $A^\dagger$ denotes the pseudo-inverse of a matrix $A$.
	As in \cite[Thm.\ 1]{eldan2020}, we find
	\begin{align*}
		\E \|\tilde{S}_n - \bar{\Y}_n\|^2 
		&\leq \frac{4}{n} \sum_{t=1}^n \int_0^\infty [\E\, \tr(\Gamma_s^t)] \wedge \tfrac{d}{n}\, ds \\
		&\leq \frac{4}{n} \sum_{t=1}^n \int_0^{4\beta_t^2 \log_2(n)} \frac{d}{n}\, ds + \int_{4\beta_t^2 \log_2(n)}^\infty d\cdot P(\tau_t > s)\, ds \\
		&= \frac{C d \log(n) \sum_{t=1}^n \beta_t^2}{n^2} + \frac{4}{n} \sum_{t=1}^n 2\beta_t^2 \int_{2\log_2(n)}^\infty d \cdot P(\tau_t > s\cdot 2\beta_t^2)\, ds\\
		&\leq  \frac{C d \log(n) \sum_{t=1}^n \beta_t^2}{n^2} + \frac{4d}{n} \sum_{t=1}^n 2\beta_t^2 \int_{2\log_2(n)}^\infty \frac{1}{2^{s-1}}\, ds\\
		&\leq \frac{C d \log(n) \sum_{t=1}^n \beta_t^2}{n^2},
	\end{align*}
	for some universal $C$, which may change from line to line.
	
	By potentially changing the underlying probability space again, we may also decompose $\bar{\Y}_n = \frac{1}{\sqrt{n}}\sum_{t=1}^n \bar{Y}_t$ for independent random vectors $\bar{Y}_t\sim\mathcal{N}(0,\Cov(\bar{X}_t))$.
	For example, we can construct the $Y_t$ first, sum them up to obtain $\bar{\Y}_n$, and lastly construct the $\tilde{X}_t, t=1,\ldots, n$, from the conditional distribution given $\bar{\Y}_n$. 
	
	Now denote $S_n = \frac{1}{\sqrt{n}} \sum_{t=1}^n X_t$ and $\bar{S}_n = \frac{1}{\sqrt{n}} \sum_{t=1}^n \bar{X}_t$, and note that $\bar{\Y}_n\sim\mathcal{N}(0, \Cov(\bar{S}_n))$ since $\bar{S}_n \deq \tilde{S}_n$.
	It is possible to construct independent Gaussian random vectors $Y_t^*, t=1,\ldots, n$, such that the Gaussian random vector 
	\begin{align*}
	\left(\bar{Y}_1,\ldots,\bar{Y}_n,\;Y_1^*,\ldots, Y_n^* \right) \in \R^{2nd} 
	\end{align*}
	has the same covariance structure as the non-Gaussian random vector
	\begin{align*}
		\left(\bar{X}_1,\ldots, \bar{X}_n,\; (X_1-\bar{X}_1),\ldots, (X_n-\bar{X}_n)\right) \in \R^{2nd}.
	\end{align*}
	In particular, for $\Y_n^* = \sum_{t=1}^n Y_t^*$, we have $(\bar{\Y}_n + \Y_n^*)\sim\mathcal{N}(0,\Cov(S_n))$.
	Then
	\begin{align*}
		\E\left\| \tilde{S}_n - (\bar{\Y}_n+\Y_n^*) \right\|^2 
		&\leq 2 \E \left\| \tilde{S}_n - \bar{\Y}_n \right\|^2 + 2 \E \|\Y_n^*\|^2 \\
		&=    2 \E \left\| \tilde{S}_n - \bar{\Y}_n \right\|^2 + 2 \E \|S_n - \bar{S}_n\|^2 \\
		&\leq 2 \E \left\| \tilde{S}_n - \bar{\Y}_n \right\|^2 + \frac{2}{n} \sum_{t=1}^n \E \|X_t - \bar{X}_t\|^2  \\
		&\leq \frac{C d \log(n) \sum_{t=1}^n \beta_t^2}{n^2} + \frac{1}{n}\frac{4}{q-2} \sum_{t=1}^n b_t^q \beta_t^{2-q}.
	\end{align*}
	Now choose $\beta_t = b_t n^\frac{1}{q} d^{-\frac{1}{q}}$ to obtain the desired approximation rate.
	Setting $Y_t = \bar{Y}_t + Y_t^* \sim\mathcal{N}(0, \Cov(X_t))$ completes the proof.
\end{proof}

\begin{proof}[Proof of Theorem \ref{thm:eldan-sequential}]
	In the sequel, $C=C(q)$ denotes a deterministic factor depending on $q$ only, whose value might change from line to line.
	
	We split the sum into blocks of size $L\leq n$, and let $M=\lceil\frac{n}{L}\rceil$. 
	The block size $L$ will be specified later.
	Introduce the blocks
	\begin{align*}
		S^j = \sum_{t=(j-1)L+1}^{jL \wedge n} X_t,\qquad j=1,\ldots, M.
	\end{align*}
	By virtue of Theorem \ref{thm:wasserstein-eldan}, on a richer probability space, there exist Gaussian random vectors $Y_t \sim\mathcal{N}(0,\Cov(X_t))$ and independent random vectors $\tilde{X}_t\deq X_t$, $t=(j-1)L+1,\ldots, jL$ such that 
	\begin{align*}
		\E \left\| \frac{1}{\sqrt{n}}\sum_{t=(j-1)L+1}^{jL \wedge n} (\tilde{X}_t - Y_t)	\right\|^2 
		&\leq  \frac{C}{(q-2)\wedge 1} \left( \frac{d}{L} \right)^{1-\frac{2}{q}} \frac{L}{n}   B^2 \log(n).
	\end{align*}
	We may assume that these random vectors for $j=1,\ldots, M$, are defined on the same (product-)probability space, because the $S^j$ are independent.
	Then Doob's maximal inequality yields
	\begin{align*}
		&\E\left[ \max_{r=1,\ldots, M} \left\| \frac{1}{\sqrt{n}}\sum_{j=1}^r \sum_{t=(j-1)L+1}^{jL \wedge n} (\tilde{X}_t - Y_t)	\right\|^2 \right]\\
		&\leq 4\, \E \left\| \frac{1}{\sqrt{n}}\sum_{j=1}^M\sum_{t=(j-1)L+1}^{jL \wedge n} (\tilde{X}_t - Y_t) \right\|^2 \\
		&= 4 \sum_{j=1}^M \E\left\| \frac{1}{\sqrt{n}}\sum_{t=(j-1)L+1}^{jL \wedge n} (\tilde{X}_t - Y_t) \right\|^2 \\
		&\leq  C_q \frac{M\, L}{n}  B^2  \left(\frac{d}{L}\right)^{1-\frac{2}{q}} \log(n) \\
		&\leq  C_q  B^2  \left(\frac{d}{L}\right)^{1-\frac{2}{q}} \log(n).
	\end{align*}

	Now note that $\E \|Y_t\|^q \leq C (\E \|Y_t\|^2 )^\frac{q}{2}=C (\E \|\tilde{X}_t\|^2 )^\frac{q}{2}  \leq C \E \|\tilde{X}_t\|^q$ for some $C=C(q)$, see Lemma \ref{lem:gauss-norm}.
	Hence, for all $j=1,\ldots, M$, the Rosenthal inequality (Theorem \ref{thm:rosenthal}) yields
	\begin{align*}
		\left(\E\left[ \max_{(j-1)L+1 \leq k\leq jL} \left\|  \sum_{t=(j-1)L+1}^{k\wedge n} \tilde{X}_t \right\|^q\right) \right]^\frac{1}{q}
		&\leq  C \sqrt{L} B, \\
		\left(\E\left[ \max_{(j-1)L+1 \leq  k\leq jL} \left\|  \sum_{t=(j-1)L+1}^{k\wedge n} Y_t \right\|^q\right) \right]^\frac{1}{q}
		&\leq C \sqrt{L} B.
	\end{align*}
	Thus,
	\begin{align*}
		& \left(\E \max_{k=1,\ldots,n} \left\| \frac{1}{\sqrt{n}}\sum_{t=1}^k (\tilde{X}_t-Y_t)  \right\|^2\right)^\frac{1}{2} \\
		&\leq C B \sqrt{\log n} \left(\frac{d }{L}\right)^{\frac{1}{2}-\frac{1}{q}}
		\quad + \frac{1}{\sqrt{n}} \left(\E\left[ \max_{j=1,\ldots,M} \max_{(j-1)L+1 \leq k\leq jL} \left\| \sum_{t=(j-1)L+1}^{k\wedge n} (\tilde{X}_t - Y_t) \right\|^q\right) \right]^\frac{1}{q} \\
		&\leq C B \sqrt{\log n} \left(\frac{d }{L}\right)^{\frac{1}{2}-\frac{1}{q}}
		\; + \frac{1}{\sqrt{n}} M^{\frac{1}{q}}\max_{j=1,\ldots,M}\left(\E\left[  \max_{(j-1)L+1 \leq k\leq jL} \left\| \sum_{t=(j-1)L+1}^{k\wedge n} (\tilde{X}_t - Y_t) \right\|^q\right) \right]^\frac{1}{q} \\
		&\leq C B \sqrt{\log n} \left(\frac{d }{L}\right)^{\frac{1}{2}-\frac{1}{q}}
		\quad + C M^{\frac{1}{q}} \sqrt{\frac{L}{n}} B \\
		&\leq C  B \sqrt{\log n} \left(\frac{d }{L}\right)^{\frac{1}{2}-\frac{1}{q}}
		\quad + C M^{\frac{1}{q}-\frac{1}{2}} B. 
	\end{align*}
	In the last step, we use that $(\E \max_{i=1,\ldots, M} |\delta_i|^q)^\frac{1}{q} \leq M^\frac{1}{q} \max_{i=1,\ldots,M} (\E|\delta_i|^q)^\frac{1}{q}$ for any random variables $\delta_1,\ldots, \delta_M$.
	Now choose $L$ such that $M\approx \sqrt{n/d}$, to find that
	\begin{align*}
		\left(\E \max_{k=1,\ldots,n} \left\| \frac{1}{\sqrt{n}}\sum_{t=1}^k (\tilde{X}_t-Y_t)  \right\|^2\right)^\frac{1}{2}
		&\leq C B \sqrt{\log(n)} \left( \frac{d}{n} \right)^{\frac{1}{4}-\frac{1}{2q}}.
	\end{align*}
\end{proof}

\begin{proof}[Proof of Theorem \ref{thm:zaitsev}]
	In the sequel, $C$ denotes a deterministic factor, whose value might change from line to line.
	
	We will construct $\tilde{X}_t$ and $Y_t$ such that
	\begin{align}
		P\left( \max_{k=1,\ldots, n} \left\|  \sum_{t=1}^k (\tilde{X}_t-Y_t) \right\| >\tau d^{\frac{15}{2}+\alpha} L_q^\frac{1}{q} \right) 
		&\leq C \tau^{-q}. \label{eqn:zaitsev}
	\end{align}
	This inequality suffices to bound the moment of order $\eta<q$.
	To establish \eqref{eqn:zaitsev}, we apply Corollary 3 of \cite{zaitsev2007} as follows.
	By assumption \eqref{eqn:ass-matrix}, for $l_t=\lambda_{\min}(\Sigma_t)$,
	\begin{align*}
		l_t\|v\|^2 \leq v^T \Sigma_t v \leq l_t c \|v\|^2,\qquad v\in\R^d.
	\end{align*}
	Then $l_t \leq \E \|X_t\|^2 \leq (\E \|X_t\|^q)^\frac{2}{q}\leq L_q^\frac{2}{q}$. 
	We distinguish the two cases (i) $\frac{4e^2}{c} L_q^\frac{2}{q} \leq \sum_{t=1}^n l_t$, and (ii) $\frac{4e^2}{c} L_q^\frac{2}{q} > \sum_{t=1}^n l_t$.
	Note that $c\geq 1$.
	
	\underline{Case (i):}
	If $4e^2 L_q^\frac{2}{q} \leq \sum_{t=1}^n l_t$, then we can find integers $0=m_0<m_1 < \ldots < m_s=n$, $s\leq n$, such that \begin{align*}
		4e^2 L_q^\frac{2}{q} \leq \sum_{t=m_{k-1} + 1}^{m_k} l_t \leq 8e^2 L_q^\frac{2}{q}.
	\end{align*} 
	This is in particular possible because $l_t \leq L_q^\frac{2}{q}$ for all $t$.
	We may then verify condition (1.13) of \cite{zaitsev2007} with $r=2$, i.e.\
	\begin{align*}
		4e^2 L_q^\frac{2}{q} \leq \sum_{t=m_{k-1}+1}^{m_k} l_t \leq v^T \Cov\left( \sum_{t=m_{k-1}+1}^{m_k} X_t \right) v \leq \sum_{t=m_{k-1}+1}^{m_k} l_t c \leq c\cdot 8e^2 L_q^\frac{2}{q}. 
	\end{align*}
	Using $\log^*(d) \leq a d^\frac{\alpha}{2}$ for some universal factor $a=a(c)$ depending on $\alpha>0$, we have 
	\begin{align*}
		d^{\frac{15+\alpha}{2}} \log^*(d)\, r\, L_q^{\frac{1}{q}} \log^*(s) \leq \tilde{a} d^{\frac{15}{2}+\alpha} L_q^\frac{1}{q} \log^*(n), 
	\end{align*}
	for some $\tilde{a}$ depending on $c$ and $\alpha$ only.
	Corollary 3 of \cite{zaitsev2007} yields $\tilde{X}_t$ and $Y_t$ as specified, such that for some universal $\bar{a}=\bar{a}(c,\alpha)$, and some $\underline{a}=\underline{a}(c,\alpha)$, $\underline{a}^*=\underline{a}^*(c,\alpha)$, and any $z\geq \bar{a} d^{\frac{15}{2}+\alpha}  L_q^\frac{1}{q} \log^*(n)$,
	\begin{align*}
		P\left( \max_{k=1,\ldots, n} \left\|  \sum_{t=1}^k (\tilde{X}_t-Y_t) \right\| >5z \right) 
		&\leq 2 L_q z^{-q} + \exp\left( -\frac{\underline{a} z}{L_q^\frac{1}{q} d^\frac{9}{2} \log^*(d)} \right)\\
		&\leq 2 L_q z^{-q} + \exp\left( -\frac{\underline{a}^* z}{L_q^\frac{1}{q} d^{\frac{9}{2}+\alpha} } \right).
	\end{align*}
	In particular, for $\tau\geq \bar{a}$,
	\begin{align*}
		&P\left( \max_{k=1,\ldots, n} \left\|  \sum_{t=1}^k (\tilde{X}_t-Y_t) \right\| >5\tau d^{\frac{15}{2}+\alpha} L_q^\frac{1}{q} \log^*(n) \right) \\
		&\leq 2 \tau^{-q} \log^*(n)^{-q} + \exp\left( - 5 \log^*(n) \underline{a}^* \tau d^3 \right) \\
		&\leq 2\tau^{-q} + \exp(-5 \underline{a}^* \tau).
	\end{align*}
	This establishes \eqref{eqn:zaitsev}.
	
	\underline{Case (ii):}
	If $4e^2 L_q^\frac{2}{q} > \sum_{t=1}^n l_t$, the construction is simpler. 
	In this case, we may choose $\tilde{X}=X_t$ and $Y_t\sim \mathcal{N}(0,\Sigma_t)$ independent, coupled arbitrarily with the $\tilde{X}$.
	Then, for any $\tau\geq 1$,
	\begin{align*}
		&P\left( \max_{k=1,\ldots, n} \left\|  \sum_{t=1}^k (\tilde{X}_t-Y_t) \right\| >5\tau d^{\frac{15}{2}+\alpha} L_q^\frac{1}{q} \log^*(n) \right) \\
		&\leq P\left( \max_{k=1,\ldots, n} \left\|  \sum_{t=1}^k \tilde{X}_t \right\| >\tau d^7 L_q^\frac{1}{q} \right) + P\left( \max_{k=1,\ldots, n} \left\|  \sum_{t=1}^k Y_t \right\| >  \tau d^7 L_q^\frac{1}{q}\right) \\
		&\leq \frac{\E \left[\max_{k=1,\ldots,n}\left\|\sum_{t=1}^k  \tilde{X}_t \right\|^q \right]}{\tau^q  d^{7q} L_q } + \frac{\E \left[\max_{k=1,\ldots,n}\left\|\sum_{t=1}^k  Y_t \right\|^q \right]}{\tau^q  d^{7q} L_q } \\
		&\leq C \frac{\sum_{t=1}^n  \E \left\|  \tilde{X}_t \right\|^q  + \left(\sum_{t=1}^n  \E \left\|  \tilde{X}_t \right\|^2 \right)^\frac{q}{2} }{\tau^q  d^{7q} L_q } 
		+C \frac{\sum_{t=1}^n  \E \left\|  Y_t \right\|^q  + \left(\sum_{t=1}^n  \E \left\|  Y_t \right\|^2 \right)^\frac{q}{2} }{\tau^q d^{7q} L_q },  \\
	\intertext{for some $C=C(q)$ via Theorem \ref{thm:rosenthal},}
		&\leq C \frac{L_q  + \left(\sum_{t=1}^n  c\, d\, l_t \right)^\frac{q}{2} }{\tau^q d^{7q} L_q } 
		+C \frac{\sum_{t=1}^n  \E (\left\|  Y_t \right\|^2)^{\frac{q}{2}}   + \left(\sum_{t=1}^n  c\, d\, l_t \right)^\frac{q}{2} }{\tau^q  d^{7q} L_q }  \\
	\intertext{by virtue of Lemma \ref{lem:gauss-norm} since $Y_t$ is Gaussian,}
		&\leq C\frac{L_q  +\sum_{t=1}^n  l_t^\frac{q}{2} d^{\frac{q}{2}}  + \left(\sum_{t=1}^n  c\, d\, l_t \right)^\frac{q}{2} }{\tau^q  d^{7q} L_q }  \\
		&\leq C\frac{L_q  +(\sum_{t=1}^n  l_t\,d )^{\frac{q}{2}}  + \left(\sum_{t=1}^n  c\, d\, l_t \right)^\frac{q}{2} }{\tau^q  d^{7q} L_q }  \\
		&\leq C\frac{L_q  +\sum_{t=1}^n  l_t^\frac{q}{2} d^{\frac{q}{2}}  + \left(\sum_{t=1}^n  c\, d\, l_t \right)^\frac{q}{2} }{\tau^q  d^{7q} L_q }  \\
	\intertext{because $\E \|Y_t\|_q^q \leq d\, \max_{j=1,\ldots,d} \E |Y_{t,j}|^q \leq C d\, l_t$ by virtue of Lemma \ref{lem:gauss-norm},}
		&\leq C \frac{L_q + (4e^2d)^{\frac{q}{2}} L_q + (4c e^2d)^{\frac{q}{2}} L_q }{\tau^q d^{7q} L_q } \\
		&\leq C \tau^{-q}.
	\end{align*}
	This establishes \eqref{eqn:zaitsev}.
\end{proof}

\subsection{Moment bound (Theorem \ref{thm:liu})}

The following high-dimensional moment inequality may be obtained as a special case of \cite[Thm.\ 4.1]{Pinelis1994}.
\begin{theorem}[Rosenthal inequality]\label{thm:rosenthal}
	For each $2\leq r\leq q < \infty$, there exists a finite factor $C_{q,r}$ such that for any $n,d\in\N$, and any martingale-difference sequence $X_t$ in $\R^d$,
	\begin{align*}
		\left(\E \max_{k\leq n} \left\| \sum_{t=1}^k X_t \right\|_r^q \right)^\frac{1}{q}
		&\leq C_{q,r} n^{\frac{1}{2}-\frac{1}{q}} \left(\sum_{t=1}^n \E  \|X_t\|_{r}^q \right)^\frac{1}{q} \\
		&\leq C_{q,r} n^\frac{1}{2}\max_{t\leq n} (\E \|X_t\|_r^q)^\frac{1}{q} .
	\end{align*}
	If the $X_t$ are independent random vectors with $\E(X_t)=0$, then
	\begin{align*}
		\left(\E \max_{k\leq n} \left\| \sum_{t=1}^k X_t \right\|_r^q \right)^\frac{1}{q}
		&\leq C_{q,r} \left[\left(\sum_{t=1}^n \E  \|X_t\|_{r}^q \right)^\frac{1}{q}  + \left( \sum_{t=1}^n \E \|X_t\|_r^2 \right)^\frac{1}{2} \right].
	\end{align*}	
\end{theorem}
\begin{proof}[Proof of Theorem \ref{thm:rosenthal}]
	Each $X_t\in\R^d$ may be interpreted as a mapping from $S=\N$ to $\R$.
	If we endow $S$ with the counting measure, we may regard $X_t \in L_r(S)$, with $\|X_t\|_{L_r(S)}   = \|X_i\|_r$.
	Denote the canonical filtration by $\F_t = \sigma(X_1,\ldots, X_t)$.
	Then Theorem 4.1 of \cite{Pinelis1994} yields
	\begin{align}
		&\quad\left(\E \max_{k\leq n} \left\| \sum_{t=1}^k X_t \right\|_r^q \right)^\frac{1}{q} 
		= \left(\E \max_{k\leq n} \left\| \sum_{t=1}^k X_t \right\|_{L_r(S)}^q \right)^\frac{1}{q}\nonumber \\ 
		&\leq C_{q,r}\left[ \left(\E \max_{t\leq n} \|X_t\|_{L_r(S)}^q \right)^\frac{1}{q} 
		+ \left(\E\left( \sum_{t=1}^n \E(\|X_t\|_{L_r(S)}^2|\F_{t-1}) \right)^\frac{q}{2}\right)^\frac{1}{q} \right] \nonumber \\
		&\leq C_{q,r}\left[ \left(\sum_{t=1}^n \E  \|X_t\|_{r}^q \right)^\frac{1}{q} 
		+ \left(\E\left( \sum_{t=1}^n \E(\|X_t\|_{r}^2|\F_{t-1}) \right)^\frac{q}{2}\right)^\frac{1}{q} \right] \label{eqn:rosenthal-proof} \\
		&\leq C_{q,r}\left[ \left(\sum_{t=1}^n \E  \|X_t\|_{r}^q \right)^\frac{1}{q} 
		+ n^{\frac{1}{2} - \frac{1}{q}} \left(\E \sum_{t=1}^n \E(\|X_t\|_{r}^2|\F_{t-1})^\frac{q}{2}\right)^\frac{1}{q} \right] \nonumber \\
		\intertext{by applying Hölder's inequality to the second sum, with exponents $\frac{q}{2}$ and $\frac{q}{q-2}$,}
		& \leq C_{q,r} \left[ \left(\sum_{t=1}^n \E  \|X_t\|_{r}^q \right)^\frac{1}{q} 
		+ n^{\frac{1}{2} - \frac{1}{q}}\left(\E\sum_{t=1}^n \E(\|X_t\|_{r}^q|\F_{t-1})\right)^\frac{1}{q} \right] \nonumber \\
		&\leq 2 C_{q,r} n^{\frac{1}{2}-\frac{1}{q}} \left(\sum_{t=1}^n \E  \|X_t\|_{r}^q \right)^\frac{1}{q}  , \nonumber
	\end{align}
	by Jensen's inequality and since $q\geq 2$.
	If the $X_t$ are independent, we use \eqref{eqn:rosenthal-proof} and the fact that $\E (\|X_t\|^2_{L_r(S)}|\F_{t-1})=\E \|X_t\|_r^2$.
\end{proof}

\begin{proof}[Proof of Theorem \ref{thm:liu}]
	In the sequel, $C=C(q,r)$ denotes a deterministic factor, whose value might change from line to line, and which only depends on $q$ and $r$.
	
	Without loss of generality, let $\E X_{t}=0$.
	As in the proof of \cite[Thm.\ 1]{Liu2013}, define
	\begin{align*}
		X_{t,-1} &= \E (X_t) = 0,\quad t=1,\ldots, n,\\
		X_{t,j} &= \E(X_t|\epsilon_{t},\epsilon_{t-1},\ldots, \epsilon_{t-j}), \quad t=1,\ldots, n,\quad j\in\N_0 \\
		S_{k,j} &= \sum_{t=1}^k X_{t,j},\qquad k=1,\ldots,n,\; j\in\N_0.
	\end{align*}
	Moreover, since $\E\|X_t\|^q < \infty$, the martingale convergence theorem ensures that for each $k=1,\ldots, n$, there exists some $\tilde{X}_k$ such that $\E \|\tilde{X}_k - X_{k,j}\|_r^q\to 0$ as $j\to\infty$.
	The measurability of $G_k:\R^\infty\to\R^d$ with respect to the projection $\sigma$-algebra ensures that $\tilde{X}_k=X_k$.
	Hence, telescoping yields
	\begin{align*}
		S_t = \sum_{k=1}^t X_k = \sum_{j=0}^\infty (S_{t,j}-S_{t,j-1}).
	\end{align*}
	As observed by \cite{Liu2013}, for each $j$, the random vectors $(X_{n-k,j}-X_{n-k,j-1})_{k=0}^{n-1}$ are martingale differences with respect to the filtration $\F_k = \sigma(\epsilon_{n-k-j},\epsilon_{n-k-j+1},\ldots)$.
	Thus, Theorem \ref{thm:rosenthal} and Doob's maximal inequality yield
	\begin{align}
		&\left(\E \max_{t=1,\ldots,n} \|S_{t,j}-S_{t,j-1}\|_r^q\right)^\frac{1}{q} \nonumber \\
		&\leq \left(\E \|S_{n,j}-S_{n,j-1}\|_r^q\right)^\frac{1}{q} + \left(\E \max_{t=1,\ldots,n} \|(S_{n,j}-S_{n,j-1}) -(S_{t,j}-S_{t,j-1})\|_r^q\right)^\frac{1}{q} \nonumber \\
		&\leq C \left(\E \|S_{n,j}-S_{n,j-1}\|_r^q\right)^\frac{1}{q} \nonumber \\
		&\leq 2 C n^{\frac{1}{2}-\frac{1}{q}} \left( \sum_{t=1}^n \E \|X_{t,j}-X_{t,j-1}\|_r^q \right)^\frac{1}{q}. \label{eqn:backwards-martingale}
	\end{align}
	The latter term may be bounded as
	\begin{align}
		 \E \|X_{t,j}-X_{t,j-1}\|_r^q 
		 &= \E \left\| \E \left[ G_t(\beps_t) - G_t(\tilde{\beps}_{t,t-j}) | \epsilon_t,\ldots, \epsilon_{t-j} \right] \right\|_r^q \nonumber \\
		 &\leq \E \left\| G_t(\beps_t) - G_t(\tilde{\beps}_{t,t-j}) \right\|_r^q 
		 \quad \leq \theta_{t,j,q,r}^q. \label{eqn:liu-theta}
	\end{align}
	Hence, we find that
	\begin{align*}
		\left( \E \max_{t=1,\ldots,n} \|S_t\|_r^q \right)^\frac{1}{q}
		&\leq \sum_{j=0}^\infty \left( \E \max_{t=1,\ldots,n} \|S_{t,j}-S_{t,j-1}\|_r^q \right)^\frac{1}{q} \\
		&\leq C n^{\frac{1}{2}-\frac{1}{q}} \sum_{j=0}^\infty \left(  \sum_{t=1}^n \theta_{t,j,q,r}^q\right)^\frac{1}{q} .
	\end{align*}
	This establishes \eqref{eqn:liu-1}.
	
	To establish \eqref{eqn:liu-2}, we may proceed as in \cite[Thm.\ 1]{Liu2013}, replacing (2.3) and (2.4) therein by Theorem \ref{thm:rosenthal}.
	For completeness, we repeat the argument.
	Introduce
	\begin{align*}
		Y_{i,j} &= \sum_{t=(i-1)j+1}^{(ij)\wedge n} (X_{t,j}-X_{t,j-1}), \quad i = 1,\ldots, \left(\lceil \tfrac{n}{j}\rceil \wedge n\right),\, j\in\N_0.
	\end{align*}
	The central observation is that $Y_{1,j}, Y_{3,j},\ldots$, is a sequence of independent random vectors, and so is $Y_{2,j}, Y_{4,j},\ldots$.
	Now decompose
	\begin{align*}
		\left(\E \|S_{n,j}-S_{n,j-1}\|_2^q \right)^\frac{1}{q} 
		&= \left(\E \left\|\sum_{i=1}^{\lceil \frac{n}{j} \rceil} Y_{i,j}\right\|_2^q \right)^\frac{1}{q} 
		&\leq \left( \sum_{i \text{ is odd}} \E \|Y_{i,j}\|_2^q \right)^\frac{1}{q} + \left( \sum_{i \text{ is even}} \E \|Y_{i,j}\|_2^q \right)^\frac{1}{q}.
	\end{align*}
	Hence, by virtue of Theorem \ref{thm:rosenthal},
	\begin{align*}
		\left(\E \|S_{n,j}-S_{n,j-1}\|_2^q \right)^\frac{1}{q} 
		&\leq C \left[ \left( \sum_{i \text{ is odd}} \E \|Y_{i,j}\|_2^2 \right)^\frac{1}{2} + \left( \sum_{i \text{ is even}} \E \|Y_{i,j}\|_2^2 \right)^\frac{1}{2} \right. \\
		&\left. \qquad + \left( \sum_{i \text{ is odd}} \E \|Y_{i,j}\|_2^q \right)^\frac{1}{q} + \left( \sum_{i \text{ is even}} \E \|Y_{i,j}\|_2^q \right)^\frac{1}{q} \right]\\
		&\leq C \left[ \left( \sum_{i} \E \|Y_{i,j}\|_2^2 \right)^\frac{1}{2} + \left( \sum_{i} \E \|Y_{i,j}\|_2^q \right)^\frac{1}{q} \right].
	\end{align*}
	Applying Theorem \ref{thm:rosenthal} for martingales again, and \eqref{eqn:liu-theta}, we obtain
	\begin{align*}
		\left(\E \|Y_{i,j}\|_2^q\right)^\frac{1}{q} 
		&\leq C j^{\frac{1}{2}-\frac{1}{q}} \left( \sum_{t=(i-1)j+1}^{(ij)\wedge n} \E \|X_{t,j}-X_{t,j-1}\|_2^q \right)^\frac{1}{q} \\
		&\leq C j^{\frac{1}{2}-\frac{1}{q}} \left( \sum_{t=(i-1)j+1}^{(ij)\wedge n} \theta_{t,j,q,2}^q \right)^\frac{1}{q}, \\
		\left(\E \|Y_{i,j}\|_2^2\right)^\frac{1}{2} 
		&\leq C \left( \sum_{t=(i-1)j+1}^{(ij)\wedge n} \theta_{t,j,2,2}^2 \right)^\frac{1}{2},
	\end{align*}
	so that,
	\begin{align*}
		\left(\E \|S_{n,j}-S_{n,j-1}\|_2^q \right)^\frac{1}{q} 
		&\leq C \left[ j^{\frac{1}{2}-\frac{1}{q}} \left( \sum_{t=1}^n \theta_{t,j,q,2}^q \right)^\frac{1}{q} + \left( \sum_{t=1}^n \theta_{t,j,2,2}^2 \right)^\frac{1}{2} \right].
	\end{align*}
	Using \eqref{eqn:liu-1}, we may show that
	\begin{align*}
		\left(\E \max_{k\leq n} \|S_{k}-S_{k,n}\|_2^q\right)^\frac{1}{q}
		&\leq C n^{\frac{1}{2}-\frac{1}{q}} \sum_{j=n+1}^\infty \left(\sum_{t=1}^n \theta_{t,j,q,2}^q \right)^\frac{1}{q},
	\end{align*}
	where $S_k = \sum_{t=1}^k X_t$.
	To see this, note that the physical dependence measure of $X_t - X_{t,n}$ is zero for the first $n$ lags. 
	
	Since $S_{k,0}=S_{k,0}-S_{k,-1}$, we may conclude that
	\begin{align*}
		\left(\E \max_{k\leq n} \|S_k\|_2^q\right)^\frac{1}{q} 
		&\leq \left(\E \max_{k\leq n} \|S_{k}-S_{k,n}\|_2^q\right)^\frac{1}{q} + \sum_{j=0}^n \left(\E \max_{k\leq n} \|S_{k,j}-S_{k,j-1}\|_2^q\right)^\frac{1}{q} \\
		&\leq C n^{\frac{1}{2}-\frac{1}{q}} \sum_{j=n+1}^\infty \left(\sum_{t=1}^n \theta_{t,j,q,2}^q \right)^\frac{1}{q} \\
		&\qquad +C \sum_{j=1}^n \left( j^{\frac{1}{2}-\frac{1}{q}} \left( \sum_{t=1}^n \theta_{t,j,q,2}^q \right)^\frac{1}{q} + \left( \sum_{t=1}^n \theta_{t,j,2,2}^2 \right)^\frac{1}{2} \right)  .
	\end{align*}
\end{proof}

\subsection{Blocking}

\begin{proposition}\label{prop:blocking-independent}
	Let $X_t=G(\beps_t)$, $t=1,\ldots,n$, and let $0=t_1 < \ldots < t_M= n$ be positive integers.
	For each $l=1,\ldots, M-1$, let $\tilde{\epsilon}_t^l$ be independent copies of the $\epsilon_t$, and $\bar{\beps}_{t,j}^l = (\epsilon_t,\ldots, \epsilon_{j+1}, \tilde{\epsilon}_j^l,\tilde{\epsilon}_{j-1}^l,\ldots)$, and $\tilde{X}_t = G_t(\bar{\beps}_{t, t_{l}}^l)$ for $t=t_l+1,\ldots t_{l+1}$.
	Then there exists a universal constant $C=C(q)$ such that
	\begin{align*}
		\E \left( \max_{k\leq n} \left\|\sum_{t=1}^k  (X_t - \tilde{X}_t) \right\|^q \right)^\frac{1}{q} 
		&\leq C n^{\frac{1}{2}-\frac{1}{q}} \sum_{j=0}^\infty \left( \sum_{l=1}^{M-1}\sum_{t=t_l+1}^{(t_{l}+j)\wedge t_{l+1}} \theta_{t,j,q}^q \right)^\frac{1}{q} \\
		&\leq C n^{\frac{1}{2}} \sum_{j=0}^\infty [\max_t \theta_{t,j,q}] \left( 1 \wedge \frac{M\cdot j}{n} \right)^\frac{1}{q}.
	\end{align*}
\end{proposition}
\begin{proof}[Proof of Proposition \ref{prop:blocking-independent}]
	Let $\eta_t = (\epsilon_t, \tilde{\epsilon}_t^1,\ldots, \tilde{\epsilon}_t^M)$, and for $t=1,\ldots, n$, and $j\in\N_0$,
	\begin{align*}
		X_{t,j}&=\E(X_t|\eta_t,\ldots,\eta_{t-j}), \\
		\tilde{X}_{t,j}&=\E(\tilde{X}_t|\eta_t,\ldots,\eta_{t-j}), \\
		\chi_{t,j} &= X_{t,j}-\tilde{X}_{t,j}.
	\end{align*}
	We also set $X_{t,-1}=\tilde{X}_{t,-1} = \E (X_t)$, and accordingly $\chi_{t,-1}=0$.
	For $t=t_l+1,\ldots, t_{l+1}$, and $t-j\geq t_l+1$, we have $\tilde{X}_{t,j}
	= X_{t,j}$ so that $\chi_{t,j}=0$ in this case.
	For $t-j\leq t_l$, we use that $(\tilde{X}_{t,j},\tilde{X}_{t,j-1})\deq (X_{t,j}, X_{t,j-1})$.
	Hence,
	\begin{align*}
		 \E \|\tilde{X}_{t,j}-\tilde{X}_{t,j-1}\|^q
		&= \E \|X_{t,j}-X_{t,j-1}\|^q \\ 
		&=\E \left\| \E \left[ G_t(\beps_t) - G_t(\tilde{\beps}_{t,t-j}) | \epsilon_t,\ldots, \epsilon_{t-j} \right] \right\|^q \\
		&\leq \E \left\| G_t(\beps_t) - G_t(\tilde{\beps}_{t,t-j}) \right\|^q 
		\quad \leq \theta_{t,j,q}^q.
	\end{align*}
	Hence, for $t_{l}+1\leq t \leq t_l$,
	\begin{align*}
		\E \| \chi_{t,j} - \chi_{t,j-1}\|^q \leq \begin{cases}
			0, & t-j\geq t_l+1, \\
			2^q\theta_{t,j,q}^q,& t-j\leq t_l.
		\end{cases}
	\end{align*}
	
	We may now apply the same martingale construction as in the proof of Theorem \ref{thm:liu}.
	In particular, the random vectors $(\chi_{t-k,j}-\chi_{t-k,j-1})_{k=1}^{n-1}$ are martingale differences with respect to the filtration $\F_k = \sigma(\eta_{n-k-j},\eta_{n-k-j+1},\ldots)$.
	As in \eqref{eqn:backwards-martingale}, we obtain
	\begin{align*}
		\E \left( \max_{k\leq n} \left\|\sum_{t=1}^k  (X_t - \tilde{X}_t) \right\|^q \right)^\frac{1}{q}
		&\leq \E \left( \left(\max_{k\leq n} \sum_{j=0}^\infty \left\|\sum_{t=1}^k  (\chi_{t,j}-\chi_{t,j-1}) \right\|\right)^q \right)^\frac{1}{q} \\
		&\leq \E \left( \left(\sum_{j=0}^\infty \max_{k\leq n}  \left\|\sum_{t=1}^k  (\chi_{t,j}-\chi_{t,j-1}) \right\|\right)^q \right)^\frac{1}{q} \\
		&\leq \sum_{j=0}^\infty \E \left( \max_{k\leq n}  \left\| \sum_{t=1}^k (\chi_{t,j}-\chi_{t,j-1}) \right\|^q \right)^\frac{1}{q} \\ 
		&\leq C\sum_{j=0}^\infty \E \left( \left\| \sum_{t=1}^n (\chi_{t,j}-\chi_{t,j-1}) \right\|^q \right)^\frac{1}{q} \\ 
		&\leq C n^{\frac{1}{2}-\frac{1}{q}} \sum_{j=0}^\infty \left( \sum_{t=1}^n \E \|\chi_{t,j}-\chi_{t,j-1}\|^q \right)^\frac{1}{q}, 
	\intertext{using Theorem \ref{thm:rosenthal},}
		&= C n^{\frac{1}{2}-\frac{1}{q}} \sum_{j=0}^\infty \left( \sum_{l=1}^{M-1}\sum_{t=t_l+1}^{t_{l+1}} \E \|\chi_{t,j}-\chi_{t,j-1}\|^q \right)^\frac{1}{q} \\
		&\leq C n^{\frac{1}{2}-\frac{1}{q}} \sum_{j=0}^\infty \left( \sum_{l=1}^{M-1}\sum_{t=t_l+1}^{(t_{l}+j)\wedge t_{l+1}} \theta_{t,j,q}^q \right)^\frac{1}{q} \\
		&\leq C n^{\frac{1}{2}-\frac{1}{q}} \sum_{j=0}^\infty [\max_t \theta_{t,j,q}] \left( n \wedge (M\cdot j) \right)^\frac{1}{q} \\
		&=C n^{\frac{1}{2}} \sum_{j=0}^\infty [\max_t \theta_{t,j,q}] \left( 1 \wedge \frac{M\cdot j}{n} \right)^\frac{1}{q}.
	\end{align*}
\end{proof}

\subsection{Gaussian approximation for time series (Theorem \ref{thm:Gauss-ts})}

\begin{proof}[Proof of Theorem \ref{thm:Gauss-ts}]
	In the sequel, $C$ denotes a universal constant depending on $q$ and $\beta$, which may vary from line to line.
	
	\underline{Blockwise independence:}
	Set $0=t_0<t_1 < \ldots < t_M=n$ such that $L\leq [t_{l}-t_{l-1}] \leq 2L$ for $l=1,\ldots, M$, $M=\lfloor \frac{n}{L}\rfloor$, and define $\tilde{X}_t$ as in Proposition \ref{prop:blocking-independent}. 
	Then,
	\begin{align}
		\left(\E \max_{k\leq n} \left\|\sum_{t=1}^k  (X_t - \tilde{X}_t) \right\|^2 \right)^\frac{1}{2}
		&\leq C\,\Theta n^\frac{1}{2} \left[\sum_{j=1}^{\lfloor \frac{n}{M}\rfloor}  j^{\frac{1}{2}-\beta} \left(\tfrac{M}{n}\right)^\frac{1}{2} + \sum_{j= \lfloor \frac{n}{M}\rfloor+1}^\infty j^{-\beta}  \right] \nonumber \\
		&\leq C \, \Theta n^\frac{1}{2} \left[ \left(\tfrac{n}{M}\right)^{1-\beta} + \left( \tfrac{M}{n} \right)^\frac{1}{2} \right] \nonumber \\
		&\leq C \,\Theta \left[ n^\frac{1}{2} L^{1-\beta} + \left( \tfrac{n}{L} \right)^\frac{1}{2} \right]. \label{eqn:blocking-1}
	\end{align}
	
	\underline{Bounding the moments:}
	Note that $(X_t)_{t=t_{l-1}+1}^{t_{l}} \deq (\tilde{X}_t)_{t=t_{l-1}+1}^{t_{l}}$.
	Hence, Theorem \ref{thm:liu} yields for any $l=1,\ldots, M$,
	\begin{align*}
		&\quad \left( \E \max_{k=t_{l-1}+1,\ldots, t_{l}} \left\| \sum_{t=t_{l-1}+1}^{k} \tilde{X}_t  \right\|^q\right)^\frac{1}{q}
		=\left( \E \max_{k=t_{l-1}+1,\ldots, t_{l}} \left\| \sum_{t=t_{l-1}+1}^{k} X_t  \right\|^q \right)^\frac{1}{q}  \\
		&\leq C [t_l-t_{l-1}]^\frac{1}{2} \sum_{j=1}^\infty \Theta j^{-\beta} \quad \leq C L^\frac{1}{2} \Theta,
	\end{align*} 
	since $\beta>1$.

	Now define the blockwise sums $\xi_l = \sum_{t=t_{l-1}+1}^{t_{l}} \tilde{X}_t \deq \sum_{t=t_{l-1}+1}^{t_{l}} X_t$, for $l=1,\ldots, M-1$.
	By construction, the $\xi_l$ are independent random vectors and satisfy $(\E \|\xi_l\|^q)^\frac{1}{q} \leq C \Theta L^\frac{1}{2}$. 
	Moreover, the blocking error may be controlled as
	\begin{align}
		&\left(\E \max_{k\leq n} \left\| \frac{1}{\sqrt{n}}\sum_{t=1}^k  \tilde{X}_t - \frac{1}{\sqrt{n}}\sum_{\substack{l=1 \\ t_{l} \leq k}}^{M} \xi_l\right\|^q \right)^\frac{1}{q} \nonumber \\
		&= \left(  \E \max_{l=1,\ldots, M} \max_{k=t_{l-1}+1,\ldots, t_{l}} \left\| \frac{1}{\sqrt{n}}\sum_{t=t_{l-1}+1}^{k} \tilde{X}_t  \right\|^q \right)^\frac{1}{q} \nonumber\\
		&\leq \left( \sum_{l=1}^{M} \E \max_{k=t_{l-1}+1,\ldots, t_{l}} \left\| \frac{1}{\sqrt{n}}\sum_{t=t_{l-1}+1}^{k} \tilde{X}_t  \right\|^q \right)^\frac{1}{q} \nonumber\\
		& \leq  M^\frac{1}{q} \max_{l=1,\ldots,M} \left( E \max_{k=t_{l-1}+1,\ldots, t_{l}} \left\| \frac{1}{\sqrt{n}}\sum_{t=t_{l-1}+1}^{k} \tilde{X}_t  \right\|^q \right)^\frac{1}{q} \nonumber \\
		&\leq C \Theta \left(L/n\right)^{\frac{1}{2}-\frac{1}{q}}. \label{eqn:blocking-2}
	\end{align}
	
	\underline{Gaussian approximation:}
	Since the $\xi_l$ are independent, we may apply Theorem \ref{thm:eldan-sequential} with $B=C L^\frac{1}{2} \Theta$, and sample size $M$.
	Hence, on a potentially different probability space, there exist random vectors $(\xi'_l)_{l=1}^M \deq  (\xi_l)_{l=1}^M$ and independent Gaussian random vectors $Y_l \sim \mathcal{N}(0, \Cov(\xi_l))$, $l=1,\ldots, M$, such that
	\begin{align}
		\left(\E \max_{k=1,\ldots,M} \left\| \frac{1}{\sqrt{n}} \sum_{l=1}^k (\xi'_l - Y_l) \right\|^2\right)^\frac{1}{2}
		\leq C \Theta \sqrt{\log(M)} \left( \frac{d}{M} \right)^{\frac{1}{4}-\frac{1}{2q}}, \label{eqn:blocking-3}
	\end{align} 
	for some $C$ depending on $q$.
	Upon possibly enlarging the probability space by an independent uniform random variable \citep[Lem.\ 21.1]{Billingsley1999}, we may assume that there exist random vectors $X_t',\tilde{X}_t',$ such that $\{ X_t', \tilde{X}_t', \xi'_l : t=1,\ldots,n; \, l=1,\ldots,M \} \deq \{ X_t, \tilde{X}_t, \xi_l : t=1,\ldots,n; \, l=1,\ldots,M \}$.
	On this probability space, we may combine \eqref{eqn:blocking-1}, \eqref{eqn:blocking-2}, and \eqref{eqn:blocking-3}, to obtain, 
	\begin{align}
		\left(\E \max_{k\leq n} \frac{1}{\sqrt{n}}\left\|\sum_{t=1}^k  X'_t - \sum_{\substack{l=1 \\ t_{l} \leq k}}^{M} Y_l\right\|^2 \right)^\frac{1}{2}
		&\leq C \Theta \sqrt{\log(n)} \left( \frac{d L}{n} \right)^{\frac{1}{4}-\frac{1}{2q}} + C\Theta \left( L^{1-\beta} + L^{-\frac{1}{2}} \right). \label{eqn:blocking-4}
	\end{align}
	Notice that the term \eqref{eqn:blocking-2} is negligible compared to \eqref{eqn:blocking-3}. 
	
	For each $l$, we can construct on the same probability space independent normal random vectors $Y_t' \sim\mathcal{N}(0, \Cov(\xi_l)/[t_l - t_{l-1}])$ such that $\sum_{t=t_{l-1}+1}^{t_l} Y_t' = Y_l$.
	Doob's maximal inequality and the Gaussianity imply that, 
	\begin{align*}
		\left(\E \max_{k=t_{l-1}+1,\ldots, t_l} \left\|\sum_{t=t_{l-1}+1}^k Y_t'\right\|^q\right)^\frac{1}{q}
		&\leq C (\E \|Y_l\|^q)^\frac{1}{q} 
		\leq C (\E \|Y_l\|^2)^\frac{1}{2}
		\leq C (\E \|\xi_l\|^q)^\frac{1}{q}
	\end{align*}
	see Lemma \ref{lem:gauss-norm}.
	Hence, 
	\begin{align*}
		\left( \E \max_{k\leq n} \left\|\frac{1}{\sqrt{n}}\sum_{t=1}^k Y_t' - \frac{1}{\sqrt{n}}\sum_{\substack{l=1 \\ t_{l} \leq k}}^{M} Y_l \right\|^2 \right)^\frac{1}{2} 
		&\leq \left( \sum_{l=1}^{M} \E \max_{k=t_{l-1}+1,\ldots, t_{l}} \left\| \frac{1}{\sqrt{n}}\sum_{t=t_{l}+1}^{k} Y_t'  \right\|^q \right)^\frac{1}{q} \\
		&\leq C \Theta n^{-\frac{1}{2}} L^\frac{1}{2} M^\frac{1}{q}
		\quad \leq C \Theta (L/n)^{\frac{1}{2}-\frac{1}{q}},
	\end{align*}
	which is identical to the bound \eqref{eqn:blocking-2}.
	Proceeding as in \eqref{eqn:blocking-4}, we find that
	\begin{align}
		&\left( \E \max_{k\leq n} \left\|\frac{1}{\sqrt{n}}\sum_{t=1}^k (X_t' -Y_t') \right\|^2 \right)^\frac{1}{2} \nonumber \\
		&\leq C \Theta \sqrt{\log(n)} \left( \frac{d L}{n} \right)^{\frac{1}{4}-\frac{1}{2q}} + C\Theta \left( L^{1-\beta} + L^{-\frac{1}{2}} \right), \label{eqn:approx-rate-r2}
	\end{align}
	If $\beta\geq\frac{3}{2}$, we match the rates by choosing $L = \lceil (n/d)^\frac{q-2}{3q-2} \rceil$.
	This choice of $L$ yields
	\begin{align*}
		\eqref{eqn:approx-rate-r2} 
		& \leq C \Theta \sqrt{\log(n)}  \left( \frac{d}{n} \right)^{\frac{q-2}{6q-4}},
	\end{align*}
	If $1\leq \beta <\frac{3}{2}$, we match the rates by choosing $L=\lceil (n/d)^{\frac{q-2}{4q\beta-3q-2}} \rceil \geq 1$. 
	This choice yields
	\begin{align*}
		\eqref{eqn:approx-rate-r2}
		&\leq C \Theta \sqrt{\log(n)}  \left( \frac{d}{n} \right)^{\frac{(\beta-1)(q-2)}{q(4\beta-3)-2}}.
	\end{align*}
	Combining both cases yields
	\begin{align*}
		\eqref{eqn:approx-rate-r2}
		&\leq C \Theta \sqrt{\log(n)}  \left[ \left( \frac{d}{n} \right)^{\frac{q-2}{6q-4}} +  \left( \frac{d}{n} \right)^{\frac{(\beta-1)(q-2)}{q(4\beta-3)-2}} \right].
	\end{align*}
	This establishes the claim \eqref{eqn:Gauss-ts-1}.

	\underline{Explicit covariances:}
	Note that for $t=t_{l-1}+1,\ldots, t_l$, the covariance matrices of the independent Gaussian random vectors $Y_t'$ are given by
	\begin{align}
		\Cov(Y_t')
		&= \frac{1}{t_l-t_{l-1}} \sum_{r,s=t_{l-1}+1}^{t_l} \Cov (G_r(\beps_r), G_s(\beps_s)), \label{eqn:Y-ts-cov}
	\end{align}
	which is different from the local long-run covariance matrix 
	\begin{align*}
		\Sigma_t 
		= \sum_{h\in \Z} \Cov (G_t(\beps_0), G_t(\beps_h))
		&= \frac{1}{t_l-t_{l-1}} \sum_{r=t_{l-1}+1}^{t_l} \sum_{s=-\infty}^\infty \Cov (G_t(\beps_r), G_t(\beps_s)) \\
		&= \frac{1}{t_l-t_{l-1}} \sum_{r=1}^{t_l-t_{l-1}} \sum_{s=-\infty}^\infty \Cov (G_t(\beps_r), G_t(\beps_s)).
	\end{align*}
	To establish \eqref{eqn:Gauss-ts-2}, we want to find a different approximation in terms of some independent random vectors $Y_t^*\sim\mathcal{N}(0,\Sigma_t)$.
	To this end, for $t=t_{l-1}+1,\ldots, t_l$, and $l=1,\ldots, M$, define the matrices
	\begin{align*}
		\Sigma_t^l 
		&= \frac{1}{t_l-t_{l-1}} \sum_{r,s=t_{l-1}+1}^{t_l} \Cov (G_t(\beps_r), G_t(\beps_s)) \\
		&= \frac{1}{t_l-t_{l-1}} \sum_{r,s=1}^{t_l-t_{l-1}} \Cov (G_t(\beps_r), G_t(\beps_s)).
	\end{align*}
	By Proposition \ref{prop:cov},
	\begin{align}
		\|\Sigma_t - \Sigma_t^l\|_\tr 
		&\leq \frac{1}{t_{l}-t_{l-1}} \sum_{r=1}^{t_l-t_{l-1}} \sum_{\substack{s\leq 0 \text{ or} \\ s>t_{l}-t_{l-1}}} \left\|\Cov(G_t(\beps_r), G_t(\beps_s))\right\|_\tr \nonumber \\
		&\leq  \frac{C \Theta^2}{t_{l}-t_{l-1}} \sum_{r=1}^{t_l-t_{l-1}} \sum_{\substack{s\leq 0 \text{ or} \\ s>t_{l}-t_{l-1}}}  |r-s|^{1-\beta}\nonumber \\
		&\leq  \frac{C\,\Theta^2}{t_{l}-t_{l-1}} \sum_{r=1}^{t_l-t_{l-1}} r^{2-\beta} + |r-(t_l-t_{l-1})|^{2-\beta}\nonumber 
	\intertext{since $\beta>2$,}
		&\leq \frac{C\,\Theta^2}{t_{l}-t_{l-1}} (|t_l-t_{l-1}|^{3-\beta} + 1) \nonumber \\
		&\leq C \Theta^2 L^{(2-\beta) \vee (-1)}. \label{eqn:cov-trace-1}
	\end{align}
	Moreover,
	\begin{align*}
		&\|\Sigma^l_t - \Cov(Y_t')\|_\tr \\
		&= \frac{1}{t_l-t_{l-1}} \left\| \sum_{r,s=t_{l-1}+1}^{t_l} \Cov (G_t(\beps_r), G_t(\beps_s)) - \Cov (G_r(\beps_r), G_s(\beps_s))   \right\|_\tr \\
		&= \frac{1}{t_l-t_{l-1}} \left\| \sum_{r,s=t_{l-1}+1}^{t_l} \Cov (G_t(\beps_r)-G_r(\beps_r), G_t(\beps_s)) + \Cov (G_r(\beps_r), G_t(\beps_s) - G_s(\beps_s))   \right\|_\tr \\
		&= \frac{1}{t_l-t_{l-1}} \left\| \sum_{r,s=t_{l-1}+1}^{t_l} \Cov (G_t(\beps_r)-G_r(\beps_r), G_t(\beps_s)) + \Cov (G_s(\beps_s), G_t(\beps_r) - G_r(\beps_r))   \right\|_\tr \\
		&\leq \frac{1}{t_l-t_{l-1}} \sum_{r,s=t_{l-1}+1}^{t_l}\left\|  \Cov (G_t(\beps_r)-G_r(\beps_r), G_t(\beps_s) + G_s(\beps_s)) \right\|_\tr .
	\end{align*}
	By virtue of \eqref{eqn:ass-ergodic}, we have $\left\|  \Cov (G_t(\beps_r)-G_r(\beps_r), G_t(\beps_s) + G_s(\beps_s)) \right\|_\tr \leq \Theta^2 |r-s|^{1-\beta}$, as in Proposition \ref{prop:cov}.
	To obtain an alternative bound, we may apply the identity $\|v w^T\|_\tr = \|v\| \|w\|$ (Lemma \ref{lem:trace-bound}) for $v,w\in\R^d$, and the Cauchy-Schwarz inequality, to obtain
	\begin{align*}
		&\left\|  \Cov (G_t(\beps_r)-G_r(\beps_r), G_t(\beps_s) + G_s(\beps_s)) \right\|_\tr \\
		&\leq \E \left[ \|G_t(\beps_r)-G_r(\beps_r)\| \|G_t(\beps_s) + G_s(\beps_s)\| \right] \\
		&\leq ( \E \|G_t(\beps_0)-G_r(\beps_0)\|^2)^\frac{1}{2}  ( \E \|G_t(\beps_0)+G_s(\beps_0)\|^2)^\frac{1}{2}
		\quad \leq 2 \Theta ( \E \|G_t(\beps_0)-G_r(\beps_0)\|^2)^\frac{1}{2}.
	\end{align*}
	Thus, since $\beta>2$,
	\begin{align}
		\|\Sigma^l_t - \Cov(Y_t)\|_\tr
		&\leq \frac{2\Theta}{t_l-t_{l-1}} \sum_{r,s=t_{l-1}+1}^{t_l} ( \E \|G_t(\beps_0)-G_r(\beps_0)\|^2)^\frac{1}{2} \wedge (\Theta|s-r|^{1-\beta}) \nonumber \\
		&\leq \frac{4\Theta^2}{t_l-t_{l-1}} \sum_{r=t_{l-1}+1}^{t_l} \sum_{u=1}^{2L} \left(\frac{( \E \|G_t(\beps_0)-G_r(\beps_0)\|^2)^\frac{1}{2}}{\Theta} \wedge |u|^{1-\beta} \right) \nonumber \\
		&\overset{\eqref{eqn:sum-min-bound}}{\leq} \frac{C\Theta^{1+\frac{1}{\beta-1}}}{t_l-t_{l-1}} \sum_{r=t_{l-1}+1}^{t_l}  \left[( \E \|G_t(\beps_0)-G_r(\beps_0)\|^2)^\frac{1}{2}\right]^{1-\frac{1}{\beta-1}} \nonumber \\
		&\overset{(*)}{\leq} C\Theta^{1+\frac{1}{\beta-1}}\left[ \frac{1}{t_l-t_{l-1}} \sum_{r=t_{l-1}+1}^{t_l}  ( \E \|G_t(\beps_0)-G_r(\beps_0)\|^2)^\frac{1}{2}\right]^{1-\frac{1}{\beta-1}} \nonumber \\
		&\leq C\Theta^{2} \left[\frac{1}{\Theta} \sum_{u=t_{l-1}+1}^{t_l}( \E \|G_u(\beps_0)-G_{u-1}(\beps_0)\|^2)^\frac{1}{2}\right]^{1-\frac{1}{\beta-1}}.  \label{eqn:cov-trace-2}
	\end{align}
	For the inequality $(*)$, we use Jensen's inequality for the concave function $x\mapsto x^{1-\frac{1}{\beta-1}}$.	
	Moreover, for any $\eta>1$, there exists some universal $C=C(\eta)$ such that for all $A>0$, we have
	\begin{align}
		\sum_{u=1}^\infty (u^{-\eta} \wedge A) \leq \sum_{u=1}^{\lceil A^{-\frac{1}{\eta}}\rceil } A + \sum_{u=\lceil A^{-\frac{1}{\eta}}\rceil+1}^\infty u^{-\eta} 
		\;\leq\; C A^{1-\frac{1}{\eta}}. \label{eqn:sum-min-bound}
	\end{align}
	Thus, combining \eqref{eqn:cov-trace-1} and \eqref{eqn:cov-trace-2}, we find that for $t=t_{l-1}+1,\ldots, t_l$,
	\begin{align}
		\begin{split}
			\|\Sigma_t - \Cov(Y_t')\|_\tr 
			&\leq C\Theta^2 L^{(2-\beta)\vee(-1)} \\
			&\quad + C\Theta^{2} \left[\frac{1}{\Theta} \sum_{u=t_{l-1}+1}^{t_l}( \E \|G_u(\beps_0)-G_{u-1}(\beps_0)\|^2)^\frac{1}{2}\right]^{1-\frac{1}{\beta-1}}.
		\end{split} \label{eqn:cov-frobenius-approx}
	\end{align}
	
	Now let $\Delta_t = \Sigma_t - \Cov(Y_t')$, and $|\Delta_t|$ as in Lemma \ref{lem:Gauss-cov}.
	Then there exist independent Gaussian random vectors $Y_t^*\sim\mathcal{N}(0,\Sigma_t)$ such that $Y_t^*-Y_t'$, $t=1,\ldots,n$, are also independent random vectors with $(Y_t^*-Y_t') \sim\mathcal{N}(0,|\Delta_t|)$.
	Thus, by Doob's maximal inequality, 
	\begin{align*}
		\left(\E \max_{k\leq n} \left\|\sum_{t=1}^k  (Y_t^* - Y_t') \right\|^2 \right)^\frac{1}{2} 
		\leq C \left( \sum_{t=1}^n \E \|Y_t^* - Y_t' \|^2\right)^\frac{1}{2} 
		&= C \left( \sum_{t=1}^n \tr(|\Delta_t|)\right)^\frac{1}{2} \\ 
		&= C \left( \sum_{t=1}^n  \|\Delta_t\|_\tr \right)^\frac{1}{2} .
	\end{align*}
	Now apply \eqref{eqn:cov-frobenius-approx} and assumption \eqref{eqn:ass-BV} to obtain
	\begin{align}
		&\left(\E \max_{k\leq n} \left\|\sum_{t=1}^k  (Y_t^* - Y_t') \right\|^2 \right)^\frac{1}{2} \nonumber \\
		&\leq C\Theta  n^\frac{1}{2} L^{(1-\frac{\beta}{2})\vee (-\frac{1}{2})} \nonumber \\
		&\quad + C \Theta \left(  \sum_{l=1}^M [t_{l} - t_{l-1}] \left(\frac{ 1}{\Theta} \sum_{t=t_{l-1}+1}^{t_l} (\E\|G_t(\beps_0)-G_{t-1}(\beps_0)\|^2)^\frac{1}{2}\right)^{1-\frac{1}{\beta-1}}  \right)^\frac{1}{2} \nonumber \\
		&\leq C\Theta  n^\frac{1}{2} L^{(1-\frac{\beta}{2})\vee (-\frac{1}{2})} \nonumber \\
		&\quad+ C\Theta  L^\frac{1}{2} \left(\sum_{l=1}^M \left( \frac{ 1}{\Theta} \sum_{t=t_{l-1}+1}^{t_l} \E(\|G_t(\beps_0)-G_{t-1}(\beps_0)\|^2)^\frac{1}{2}\right)^{\frac{\beta-2}{\beta-1}} \right)^\frac{1}{2} \nonumber \\
		&\overset{(**)}{\leq} C\Theta n^\frac{1}{2} L^{(1-\frac{\beta}{2})\vee (-\frac{1}{2})} \nonumber \\
		&\quad+ C\Theta L^\frac{1}{2} \left( \frac{ 1}{\Theta} \sum_{l=1}^M  \sum_{t=t_{l-1}+1}^{t_l} \E(\|G_t(\beps_0)-G_{t-1}(\beps_0)\|^2)^\frac{1}{2} \right)^{\frac{\beta-2}{\beta-1}\frac{1}{2} }  (M^{1-\frac{\beta-2}{\beta-1}})^\frac{1}{2} \nonumber \\
		&\leq C\Theta  n^\frac{1}{2} L^{(1-\frac{\beta}{2})\vee (-\frac{1}{2})} + C\Theta L^{\frac{1}{2}} \Gamma^{\frac{1}{2} \frac{\beta-2}{\beta-1}} M^{\frac{1}{2}\frac{1}{\beta-1}}. \label{eqn:approx-rate-r2-b}
	\end{align}
	At the step $(**)$, we use Hölder's inequality for the outer sum, with exponents $\frac{\beta-1}{\beta-2}>1$ and $(\beta-1)>1$.
	Combining \eqref{eqn:approx-rate-r2-b} with \eqref{eqn:approx-rate-r2}, and using $M\leq C n/L$, and $\Gamma\geq 1$, we obtain
	\begin{align}
		&\left( \E \max_{k\leq n} \left\|\frac{1}{\sqrt{n}}\sum_{t=1}^k (X_t' -Y_t^*) \right\|^2 \right)^\frac{1}{2} \nonumber \\
		&\leq C \Theta \sqrt{\log(n)} \left( \frac{d L}{n} \right)^{\frac{1}{4}-\frac{1}{2q}} + C\Theta \left( L^{1-\beta} + L^{-\frac{1}{2}} \right) \nonumber \\
		&\qquad+C\Theta  L^{(1-\frac{\beta}{2})\vee (-\frac{1}{2})} + C\Theta n^{-\frac{1}{2}} L^{\frac{1}{2}} \Gamma^{\frac{1}{2} \frac{\beta-2}{\beta-1}} M^{\frac{1}{2}\frac{1}{\beta-1}} \nonumber \\
		\begin{split}
			&\leq C\Theta \Gamma^{\frac{1}{2} \frac{\beta-2}{\beta-1}} \sqrt{\log(n)} \left\{  \left( \frac{d L}{n} \right)^{\frac{1}{4}-\frac{1}{2q}} +  L^{1-\frac{\beta}{2}} + L^{-\frac{1}{2}}  + \left(\frac{L}{n} \right)^{\frac{1}{2} \frac{\beta-2}{\beta-1}}\right\} .
		\end{split} 
		\label{eqn:approx-rate-r2-full}
	\end{align}

	Recall that $d\leq cn$ by assumption, for some $c\geq 1$.
	For any $L\in\N, L\leq c\frac{n}{d}$, we have
	\begin{align*}
		\eqref{eqn:approx-rate-r2-full} \leq
		\begin{cases}
			C\Theta  \Gamma^{\frac{1}{2} \frac{\beta-2}{\beta-1}} \sqrt{\log(n)} \left\{  \left( \frac{d L}{n} \right)^{\frac{1}{4}-\frac{1}{2q}} + L^{-\frac{1}{2}}\right\}, 
			& \beta\geq 3, \\
			C\Theta  \Gamma^{\frac{1}{2} \frac{\beta-2}{\beta-1}} \sqrt{\log(n)} \left\{  \left( \frac{d L}{n} \right)^{\frac{1}{4}-\frac{1}{2q}} + L^{1-\frac{\beta}{2}} \right\}, 
			& \frac{3+\frac{2}{q}}{1+\frac{2}{q}} < \beta < 3, \\
			C\Theta \Gamma^{\frac{1}{2} \frac{\beta-2}{\beta-1}} \sqrt{\log(n)} \left\{  L^{1-\frac{\beta}{2}} + \left( \frac{d L}{n} \right)^{\frac{1}{2} \frac{\beta-2}{\beta-1}}\right\}, 
			& 2<\beta\leq \frac{3+\frac{2}{q}}{1+\frac{2}{q}},
		\end{cases}
	\end{align*}
	where the factor $C=C(\beta,q,c)$ may also depend on $c$.\\
	If $\beta\geq 3$, we choose $L = \lfloor (n/d)^\frac{q-2}{3q-2} \rfloor$.
	This yields
	\begin{align*}
		\eqref{eqn:approx-rate-r2-full}
		& \leq C \Theta  \Gamma^{\frac{1}{2} \frac{\beta-2}{\beta-1}} \sqrt{\log(n)}  \left( \frac{d}{n} \right)^{\frac{q-2}{6q-4}}.
	\end{align*}
	If $\frac{3+\frac{2}{q}}{1+\frac{2}{q}} < \beta < 3$, we choose $L=\lfloor (n/d)^{(\frac{1}{2}-\frac{1}{q})/(\beta-\frac{3}{2}-\frac{1}{q})}\rfloor \leq \sqrt{n/d}$.
	This yields
	\begin{align*}
		\eqref{eqn:approx-rate-r2-full}
		&\leq C \Theta \Gamma^{\frac{1}{2} \frac{\beta-2}{\beta-1}} \sqrt{\log(n)} \left( \frac{d}{n} \right)^{ \frac{(\beta-2)(\frac{1}{2}-\frac{1}{q})}{2\beta-3-\frac{2}{q}} }\\
		&= C \Theta  \Gamma^{\frac{1}{2} \frac{\beta-2}{\beta-1}} \sqrt{\log(n)} \left( \frac{d}{n} \right)^{ \frac{(\beta-2)(q-2)}{(4\beta-6)q-4} }.
	\end{align*}
	If $2<\beta\leq \frac{3+\frac{2}{q}}{1+\frac{2}{q}}$, we choose $L=\lfloor (n/d)^\frac{1}{\beta} \rfloor \leq \sqrt{n/d}$.
	This yields 
	\begin{align*}
		\eqref{eqn:approx-rate-r2-full}
		&\leq C\Theta  \Gamma^{\frac{1}{2} \frac{\beta-2}{\beta-1}} \sqrt{\log(n)} \left( \frac{d}{n} \right)^{\frac{1}{2}-\frac{1}{\beta}} .
	\end{align*}
\end{proof}

\subsection{Covariance estimation}

\begin{proof}[Proof of Theorem \ref{thm:cov-est}]
	Throughout the proof, $C=C(\beta,q)$ denotes a factor whose value may change from line to line.
	
	The inequality $\|vv^T - ww^T\|_\tr \leq \|(v-w)w^T\|_\tr + \|w (v-w)^T\|_\tr \leq 2 \|v-w\| (\|v\| + \|w\|)$ for $v,w\in\R^d$ yields
	\begin{align}
	&\E \max_{k=b,\ldots,n} \left\| \hat{Q}(k) - \sum_{t=b}^k \frac{1}{b} \left( \sum_{s=t-b+1}^t G_t(\beps_s) \right)^{\otimes 2} \right\|_\tr \nonumber \\
	&= \E \max_{k=b,\ldots,n} \left\| \frac{1}{b} \sum_{t=b}^k \left( \sum_{s=t-b+1}^t G_s(\beps_s) \right)^{\otimes 2} -  \left( \sum_{s=t-b+1}^t G_t(\beps_s) \right)^{\otimes 2} \right\|_\tr \nonumber \\
	&\leq \sum_{t=b}^n \frac{2}{b}  \left(\E\left\| \sum_{s=t-b+1}^t [G_s(\beps_s) - G_t(\beps_s)] \right\|^2\right)^\frac{1}{2} 
	\left(\E\left\| \sum_{s=t-b+1}^t [G_s(\beps_s) + G_t(\beps_s)] \right\|^2\right)^\frac{1}{2} \nonumber \\
	&\leq \sum_{t=b}^n \frac{C}{b}  \left(\E\left\| \sum_{s=t-b+1}^t [G_s(\beps_s) - G_t(\beps_s)] \right\|^2\right)^\frac{1}{2} \Theta \sqrt{b} \nonumber \\
	\intertext{for some $C=C(\beta)$, using Theorem \ref{thm:liu} and Assumption \eqref{eqn:ass-ergodic} for $\beta>1$,}
	&\leq \frac{C \Theta}{\sqrt{b}}\sum_{t=b}^n   \sum_{s=t-b+1}^t \left(\E\left\|G_s(\beps_0) - G_t(\beps_0) \right\|^2\right)^\frac{1}{2} \nonumber \\
	&\leq C\Theta \sqrt{b}\sum_{t=1}^{n-1} \left(\E\left\|G_{t+1}(\beps_0) - G_t(\beps_0) \right\|^2\right)^\frac{1}{2} 
	\qquad \leq C \Theta^2 \Gamma \sqrt{b}. \label{eqn:varest-1}
	\end{align}
	
	We hence proceed to analyze
	\begin{align*}
	\widetilde{Q}(k) &= \sum_{t=b}^k \eta_t, \quad \eta_t = \frac{1}{b} \left( \sum_{s=t-b+1}^t G_t(\beps_s)\right)^{\otimes 2} \in \R^{d\times d}.
	\end{align*}
	Using Proposition \ref{prop:cov}, we find that
	\begin{align}
	\left\| \E (\eta_t) - \Sigma_t \right\|_\tr
	&= \left\|  \frac{1}{b}\sum_{s,s'=t-b+1}^t \gamma_{t}(s-s') - \sum_{h\in\Z} \gamma_t(h)  \right\|_\tr \nonumber \\
	&\leq \sum_{h\in\Z} \|\gamma_t(h)\|_\tr \left(\frac{|h|\wedge b}{b}\right) \nonumber \\
	&\leq C \Theta^2 \sum_{h\in\Z} |h|^{1-\beta} \left(\frac{|h|\wedge b}{b}\right) \nonumber \\
	&\leq C \Theta^2 (b^{2-\beta} + b^{-1}) . \label{eqn:varest-2}
	\end{align}
	Furthermore, Proposition \ref{prop:cov} yields $\|\Sigma_t\|_\tr \leq C \Theta^2$, for a universal $C$ depending on $\beta$.
	
	Moreover,  we may write $(\eta_t-\E(\eta_t)) = H_t(\beps_t) \in \R^{d\times d}$, for some measurable kernel $H:\R^\infty\to\R$.
	We want to apply Theorem \ref{thm:liu} to the process $[\widetilde{Q}(k) - \E(\widetilde{Q}(k))]$.
	To this end, we use that $\|v w^T\|_\tr = \|v w^T\|_F = \|v\|_2 \|w\|_2$ (Lemma \ref{lem:trace-bound}), and we find that for $j\geq 0$,
	\begin{align*}
	&\left(\E \|H_t(\beps_t) - H_t(\tilde{\beps}_{t,t-j})\|_F^2 \right)^\frac{1}{2} \nonumber \\
	&= \frac{1}{b}  \Bigg(\E\, \Bigg\| \left( \sum_{s=t-b+1}^t G_{t}(\beps_s)\right)\left( \sum_{s=t-b+1}^t G_{t}(\beps_s)\right)^T \nonumber \\
	&\qquad - \left( \sum_{s=t-b+1}^t G_{t}(\tilde{\beps}_{s,t-j})\right)\left( \sum_{s=t-b+1}^t G_{t}(\tilde{\beps}_{s,t-j})\right)^T\Bigg\|_F^2\Bigg)^\frac{1}{2} \nonumber \\
	&\leq \frac{1}{b} \left( \E \left\| \left( \sum_{s=t-b+1}^t G_{t}(\beps_s)\right) \left( \sum_{s=t-b+1}^t G_{t}(\beps_s) - G_{t}(\tilde{\beps}_{s,t-j})\right) \right\|_F^2 \right)^\frac{1}{2} \nonumber \\
	&\qquad + \frac{1}{b} \left( \E \left\| \left( \sum_{s=t-b+1}^t G_{t}(\tilde{\beps}_{s,t-j})\right) \left( \sum_{s=t-b+1}^t G_{t}(\beps_s) - G_{t}(\tilde{\beps}_{s,t-j})\right)^T \right\|_F^2 \right)^\frac{1}{2} \nonumber \\
	&\leq \left( \E \left\| \sum_{s=t-b+1}^t G_{t}(\beps_s)\right\|_2^4\right)^\frac{1}{4} \frac{1}{b} \sum_{s=t-b+1}^t \left(\E \| G_{t}(\beps_s) - G_{t}(\tilde{\beps}_{s,t-j})\|_2^4\right)^\frac{1}{4} \nonumber \\
	&\qquad +\left( \E \left\| \sum_{s=t-b+1}^t G_{t}(\tilde{\beps}_{s,t-j})\right\|_2^4\right)^\frac{1}{4} \frac{1}{b} \sum_{s=t-b+1}^t \left(\E \| G_{t}(\beps_s) - G_{t}(\tilde{\beps}_{s,t-j})\|_2^4\right)^\frac{1}{4} \nonumber \\
	&\leq 2 \left( \E \left\| \sum_{s=t-b+1}^t G_{t}(\beps_s)\right\|_2^4\right)^\frac{1}{4} \frac{1}{b} \sum_{s=t-b+1}^t \left(\E \| G_{t}(\beps_s) - G_{t}(\tilde{\beps}_{s,t-j})\|_2^4\right)^\frac{1}{4} \nonumber \\
	&\leq 2 \left( \E \left\| \sum_{s=t-b+1}^t G_{t}(\beps_s)\right\|_2^4\right)^\frac{1}{4} \frac{\Theta}{b} \sum_{s=(t-b+1)\vee(t-j)}^t ((s-t+j)\vee 1)^{-\beta}. \nonumber
	\end{align*}
	Here, we have slightly abused the notation to mean $\tilde{\beps}_{s,r}=\beps_s$ if $r>s$.
	By applying Theorem \ref{thm:liu} to the kernel $G_{t}$, we find that $\left( \E \left\| \sum_{s=t-b+1}^t G_{t}(\beps_s)\right\|_2^4\right)^\frac{1}{4}\leq C \Theta \sqrt{b}$.
	Hence,
	\begin{align}
	\left(\E \|H_t(\beps_t) - H_t(\tilde{\beps}_{t,t-j})\|_F^2 \right)^\frac{1}{2} 
	&\leq \frac{C\Theta^2}{\sqrt{b}} \sum_{s=(t-b+1)\vee(t-j)}^t ((s-t+j)\vee 1)^{-\beta} \nonumber \\
	&\leq \frac{C\Theta^2}{\sqrt{b}} \sum_{r=0}^{b\wedge j} ((j-r)\vee 1)^{-\beta} \nonumber \\
	&\leq \frac{C\Theta^2}{\sqrt{b}} \left[j^{1-\beta} + ((j-b)\vee 1)^{1-\beta} \right]. \label{eqn:H-phys}
	\end{align}
	Now note that the Frobenius norm of a matrix $A\in\R^{d\times d}$ equals the Euclidean vector norm if $A$ is considered as a vector $A\in\R^{d^2}$.
	Thus, the upper bound \eqref{eqn:H-phys} may be interpreted as the physical dependence measure $\theta_{t,j,2,2}$ of the process $H_t$.
	It may be simplified as
	\begin{align*}
	\theta_{t,j,2,2} \leq \begin{cases}
	\frac{C\Theta^2}{\sqrt{b}}, & j\leq 2b, \\
	\frac{C\Theta^2}{\sqrt{b}} j^{1-\beta}, & j> 2b.
	\end{cases}
	\end{align*}
	Now Theorem \ref{thm:liu} is applicable.
	Since $\beta>2$, we find that
	\begin{align}
	\left( \E \max_{k=b,\ldots,n} \left\| \widetilde{Q}(k) - \E (\widetilde{Q}(k)) \right\|_F^2 \right)^\frac{1}{2} 
	&\leq C  n^{\frac{1}{2}} \sum_{j=1}^\infty \max_{t\leq n} \theta_{t,j,2,2} \nonumber  \\
	&\leq C n^{\frac{1}{2}} \left[ \sum_{j=1}^{2b} \frac{\Theta^2}{\sqrt{b}} + \sum_{j=2b+1}^\infty \frac{\Theta^2}{\sqrt{b}} j^{1-\beta}  \right] \nonumber \\
	&\leq C \Theta^2 \left(\frac{n}{b}\right)^\frac{1}{2} \left[ 2b + (2b)^{2-\beta} \right] \nonumber \\
	&\leq C \Theta^2 \sqrt{n}\sqrt{b}. \label{eqn:varest-3}
	\end{align}
	For a matrix $A\in\R^{d\times d}$, it holds that $\|A\|_{\tr}\leq \sqrt{d} \|A\|_F$.
	Hence, we have shown the inequality $\left( \E \max_{k=b,\ldots,n} \left\| \widetilde{Q}(k) - \E (\widetilde{Q}(k)) \right\|_\tr^2 \right)^\frac{1}{2} \leq C \Theta^2 \sqrt{ndb}$.
	
	Combining \eqref{eqn:varest-1}, \eqref{eqn:varest-2}, and \eqref{eqn:varest-3}, we find that
	\begin{align*}
	&\E \max_{k=1,\ldots,n} \left\| \hat{Q}(k) - \sum_{t=1}^k \Sigma_t  \right\|_\tr \\
	&\leq C\Theta^2 \Gamma \sqrt{b} + C  \Theta^2 \sqrt{ndb} + \sum_{t=b}^n \|\E(\eta_t)-\Sigma_t\|_\tr + \sum_{t=1}^{b-1} \|\Sigma_t\|_\tr \\
	&\leq C\Theta^2 \Gamma \sqrt{b} + C  \Theta^2 \sqrt{ndb} + C n \Theta^2 (b^{2-\beta} + b^{-1}) + Cb\Theta^2 \\
	&\leq C \Theta^2 \left( \Gamma \sqrt{b} + \sqrt{ndb} + n b^{-1} + n b^{2-\beta} \right).
	\end{align*}
	For the last inequality, we use that $b\leq n$, such that $b\leq \sqrt{ndb}$.
	This completes the proof.
\end{proof}

\begin{proof}[Proof of Proposition \ref{prop:var-consistency}]
	Let $L\in \{1,\ldots, n\}$, and $M=\lceil n/L\rceil$, to be specified later.
	Denote $\xi_l = \sum_{t=(l-1)L+1}^{(lL) \wedge n} Y_t$, and $S_l=\sum_{t=(l-1)L+1}^{(lL) \wedge n} \Sigma_t$, and $S_l'=\sum_{t=(l-1)L+1}^{lL \wedge n} \Sigma_t'$, for $l=1,\ldots,M$.
	Then $\xi_l$ are independent Gaussian random vectors, $\xi_l\sim\mathcal{N}(0,S_l)$. 
	Denoting $\Delta_l = S_l - S_l'$, and $|\Delta_l|$ as in Lemma \ref{lem:Gauss-cov}, we find Gaussian random vectors $\zeta_l\sim\mathcal{N}(0,|\Delta_l|)$ such that $\xi_l' = \xi_l + \zeta_l \sim\mathcal{N}(0,S_l')$.
	We may also split $\zeta_l'$ into independent terms, i.e.\ we find independent Gaussian random vectors $Y_t'\sim\mathcal{N}(0,\Sigma_t')$ such that $\xi_l' = \sum_{t=(l-1)L+1}^{(lL) \wedge n} Y_t'$.
	This construction yields that the $(Y_t')_{t=1}^n$ and $(Y_t)_{t=1}^n$ are sequences of independent random vectors, while $Y_t'$ and $Y_{t+1}$ are not necessarily independent.
	We also introduce the notation $\zeta_s= \sum_{t=(l-1)L+1}^s Y_t$ for $s=(l-1)L+1,\ldots, (lL)\wedge n$, and $\zeta_s'$ analogously.
	Then
	\begin{align*}
		&\E \max_{k=1,\ldots,n} \left\| \sum_{t=1}^k (Y_t-Y_t')\right\|^2 \\
		&\leq \E \max_{k=1,\ldots,M} \left\| \sum_{l=1}^k (\xi_l - \xi_l')  \right\|^2 
		+ \E \max_{l=1,\ldots,M} \max_{s=(l-1)L+1,\ldots, (lL)\wedge n} \left\| \sum_{t=(l-1)L+1}^{s} (Y_t - Y_t')  \right\|^2 \\
		&\leq \sum_{l=1}^M \|\Delta_l\|_{\tr} 
		+ 2\E \max_{l=1,\ldots,M} \max_{s=(l-1)L+1,\ldots, (lL)\wedge n} \left\| \sum_{t=(l-1)L+1}^{s} Y_t \right\|^2\\
			&\quad + 2\E \max_{l=1,\ldots,M} \max_{s=(l-1)L+1,\ldots, (lL)\wedge n} \left\| \sum_{t=(l-1)L+1}^{s} Y_t' \right\|^2 \\
		&=\sum_{l=1}^M \|\Delta_l\|_{\tr} 
		+ 2\E \max_{s=1,\ldots,n} \left\| \zeta_s \right\|^2
		+ 2\E \max_{s=1,\ldots,n} \left\| \zeta_s' \right\|^2 \\
		&\leq \sum_{l=1}^M \left[ \left\|\sum_{t=1}^{(lL)\wedge n} \Sigma_t-\Sigma_t'\right\|_{\tr} + \left\|\sum_{t=1}^{(l-1)L} \Sigma_t-\Sigma_t'\right\|_{\tr}\right]
		+ 2\E \max_{s=1,\ldots,n} \left\| \zeta_s \right\|^2
		+ 2\E \max_{s=1,\ldots,n} \left\| \zeta_s' \right\|^2 \\
		&\leq 2 M \max_{k=1,\ldots,n} \left\| \sum_{t=1}^k (\Sigma_t-\Sigma_t') \right\|_\tr
		+ 2\E \max_{s=1,\ldots,n} \left\| \zeta_s \right\|^2
		+ 2\E \max_{s=1,\ldots,n} \left\| \zeta_s' \right\|^2 .
	\end{align*} 

	Since the random vectors $\zeta_s$ are Gaussian, the random variable $\|\zeta_s\|^2$ is sub-exponential with sub-exponential norm bounded by $C\, \tr(\Cov(\zeta_s)) \leq C\, \tr( S_l )$, for some universal factor $C$, and for $s=(l-1)L+1,\ldots, (lL)\wedge n$.
	To see this, denote the sub-exponential norm by $\|\cdot\|_{\psi_1}$.
	Then 
	\begin{align*}
		\| \, \|\zeta_s\|_2^2\,\|_{\psi_1} 
		= \left\| \sum_{j=1}^d \sigma_j^2 \delta_j^2\right\|_{\psi_1} \leq \sum_{j=1}^d \sigma_j^2 \|\delta_j^2\|_{\psi_1},
	\end{align*}
	where $\sigma_j^2$ are the eigenvalues of $\Cov(\zeta_s)$, and $\delta_j\sim\mathcal{N}(0,1)$.
	A consequence of this sub-exponential bound is that, for a potentially larger $C$, 
	\begin{align*}
		\E \max_{s=1\ldots, n} \left\|\zeta_s \right\|^2 \leq C \log(n) \max_l \|S_l\|_{\tr}
		\leq C \log(n) L\rho.
	\end{align*}
	Analogously,
	\begin{align*}
		\E \max_{s=1,\ldots,n} \left\|\zeta_s' \right\|^2 
		\leq  C \log(n) \max_l \|S_l'\|_{\tr}
		&\leq C \log(n) \max_l \left[\|S_l\|_{\tr} + \delta \right] \\
		&\leq C \log(n)  \left[L \rho + \delta \right],
	\end{align*}
	such that
	\begin{align*}
		\E \max_{k=1,\ldots,n} \left\| \sum_{t=1}^k (Y_t-Y_t')\right\|^2 
		&\leq 2 M \delta + C \log(n)\left[L \rho + \delta  \right] \\
		&\leq C \log(n) \left[ \frac{n}{L} \delta + L \rho \right].
	\end{align*}
	To minimize this expression, choose $L=1 \vee \lceil\sqrt{n \delta/\rho}\rceil$. 
	Since $\delta\leq n\rho$, we have $L\leq n$.
	This choice of $L$ yields the desired upper bound.
\end{proof}

\begin{proof}[Proof of Proposition \ref{prop:seq-test}]
	Let $Y_t^* \sim \mathcal{N}(0, \Sigma_t), t=1,\ldots,n$, be the Gaussian approximation as in \eqref{eqn:Gauss-ts-2}.
	In view of \eqref{eqn:Tn-continuity}, we find that
	\begin{align}
		&P\left( T_n(X_1,\ldots, X_n) > a_{\alpha-\nu_n}(\hat{Q}) + \tau_n \right) \nonumber \\
		&\leq P\left( T_n(Y_1^*,\ldots, Y_n^*) > a_{\alpha-\nu_n}(\hat{Q}) + \tfrac{\tau_n}{2} \right)
		+ P\left( \max_{k=1,\ldots,n} \left\| \frac{1}{\sqrt{n}} \sum_{t=1}^k (X_t-Y_t^*) \right\| > \tfrac{\tau_n}{2} \right)\nonumber \\
		&\leq P\left( T_n(Y_1^*,\ldots, Y_n^*) > a_{\alpha-\nu_n}(\hat{Q}) + \tfrac{\tau_n}{2} \right) + \frac{C\Theta_n \sqrt{\log(n)} \left(\tfrac{d_n}{n}\right)^{\xi(q,\beta)}}{\tau_n}\nonumber \\
		&\leq P\left( T_n(Y_1^*,\ldots, Y_n^*) > a_{\alpha}(Q) \right) + P\left( a_{\alpha}(Q) > a_{\alpha-\nu_n}(\hat{Q}) + \tfrac{\tau_n}{2} \right)
		+ \frac{C\Theta_n \sqrt{\log(n)} \left(\tfrac{d_n}{n}\right)^{\xi(q,\beta)}}{\tau_n}\nonumber \\
		&\leq \alpha
		+ P\left( a_{\alpha}(Q) > a_{\alpha-\nu_n}(\hat{Q}) + \tfrac{\tau_n}{2} \right)
		+ \frac{C\Theta_n \sqrt{\log(n)} \left(\tfrac{d_n}{n}\right)^{\xi(q,\beta)}}{\tau_n}. \label{eqn:seq-test-1}
	\end{align}
	It remains to bound the probability $P\left( a_{\alpha}(Q) > a_{\alpha-\nu_n}(\hat{Q}) + \tfrac{\tau_n}{2} \right)$.
	To this end, we employ Proposition \ref{prop:var-consistency}.
	For some cumulative covariance process $\overline{Q}$, let $\overline{\Sigma}_t = \overline{Q}_t - \overline{Q}_{t-1}$, and consider independent Gaussian random vectors $Z_t\sim\mathcal{N}(0,\overline{\Sigma}_t)$, which are coupled with the $Y_t^*\sim\mathcal{N}(0,\Sigma_t)$ such that 
	\begin{align}
		\E \max_{k=1,\ldots,n}\left\| \sum_{t=1}^k Y_t^*- \sum_{t=1}^k Z_t \right\|^2 
		&\leq C \log(n) \left[\sqrt{n\delta\rho} + \rho \right] 
		\quad = \Delta_n, \label{eqn:seq-test-3}
	\end{align}
	with $\rho$ and $\delta$ as in Proposition \ref{prop:var-consistency}.
	Then
	\begin{align*}
		&P\left(T_n(Y_1^*,\ldots, Y_n^*)> a_{\alpha-\nu_n}(\overline{Q}) + \tfrac{\tau_n}{2}\right) \\
		&\leq P\left(T_n(Z_1,\ldots, Z_n)> a_{\alpha-\nu_n}(\overline{Q})\right) + P\left( \frac{1}{\sqrt{n}}\max_{k=1,\ldots,n}\left\| \sum_{t=1}^k Y_t^* - \sum_{t=1}^k Z_t \right\| > \frac{\tau_n}{2} \right) \\
		&\leq P\left(T_n(Z_1,\ldots, Z_n)> a_{\alpha-\nu_n}(\overline{Q})\right) + P\left( \frac{1}{n}\max_{k=1,\ldots,n}\left\| \sum_{t=1}^k Y_t^* - \sum_{t=1}^k Z_t \right\|^2 > \frac{\tau_n^2}{4} \right) \\
		&\leq (\alpha - \nu_n) + \frac{\Delta_n}{n} \frac{4}{\tau_n^2}  
		=\alpha + \left[ \frac{\Delta_n}{n} \frac{4}{\tau_n^2}  - \nu_n \right].
	\end{align*}
	Hence, if $\left[ \frac{\Delta_n}{n} \frac{4}{\tau_n^2}  - \nu_n \right] <0$, then $a_\alpha(Q) \leq a_{\alpha-\nu_n}(\overline{Q})+\frac{\tau_n}{2}$.
	Now employ this implication to \eqref{eqn:seq-test-1}, with $\overline{Q}=\hat{Q}$.
	We denote the corresponding error \eqref{eqn:seq-test-3} by $\hat{\Delta}_n$, which is a random variable because $\hat{Q}$ is random.
	We find that
	\begin{align}
		&P\left( T_n(X_1,\ldots, X_n) > a_{\alpha-\nu_n}(\hat{Q}) + \tau_n \right) \nonumber \\
		&\leq \alpha
		+ \frac{C\Theta_n \sqrt{\log(n)} \left(\tfrac{d_n}{n}\right)^{\xi(q,\beta)}}{\tau_n} 
		+ P\left( \hat{\Delta}_n > \frac{n \nu_n \tau_n^2}{4} \right). \label{eqn:seq-test-2}
	\end{align}
	To derive a bound on $\hat{\Delta}_n$, note that $\rho\leq C \Theta_n^2$ for some $C=C(\beta)$, by virtue of Proposition \ref{prop:cov}. 
	By Theorem \ref{thm:cov-est}, 
	\begin{align*}
		\E (\delta) = \E \max_{k=1,\ldots,n} \left\| \hat{Q}_k - Q_k \right\|_\tr
		&\leq C \Theta_n^2\left(\Gamma_n \sqrt{b_n} + \sqrt{nd_nb_n} + nb_n^{-1} + nb_n^{2-\beta}\right).
	\end{align*}
	Hence, 
	\begin{align*}
		\Delta_n = \mathcal{O}_P \left( \log(n) \Theta_n^2 \left( \Gamma_n^{\frac{1}{2}} n^{\frac{1}{2}} b_n^{\frac{1}{4}} + n^{\frac{3}{4}} d_n^\frac{1}{4} b_n^{\frac{1}{4}} + n b_n^{-\frac{1}{2}} + n b_n^{1-\frac{\beta}{2}} +1 \right) \right),
	\end{align*}
	so that $P(\hat{\Delta}_n > n \nu_n \tau_n^2 /4)\to 0$ if
	\begin{align*}
		\nu_n \tau_n^2 \quad \gg \quad
		\log(n) \Theta_n^2 \left( \Gamma_n^{\frac{1}{2}} n^{-\frac{1}{2}} b_n^{\frac{1}{4}} + n^{-\frac{1}{4}} d_n^\frac{1}{4} b_n^{\frac{1}{4}} +  b_n^{-\frac{1}{2}} +  b_n^{1-\frac{\beta}{2}} +n^{-1} \right).
	\end{align*}
	In view of \eqref{eqn:seq-test-2}, we have shown that condition \eqref{eqn:seq-test-ass} implies	\begin{align*}
		\limsup_{n\to\infty} P\left( T_n(X_1,\ldots, X_n) > a_{\alpha-\nu_n}(\hat{Q}) + \tau_n \right) \leq \alpha.
	\end{align*}

\end{proof}

\section*{Acknowledgements}
The authors gratefully acknowledge the financial support from Deutsche Forschungsgemeinschaft (DFG),
grant STE 1134-11/2.

\bibliography{strong-approximation2}

\begin{thebibliography}{}

\bibitem[Berkes et~al., 2014]{Berkes2014}
Berkes, I., Liu, W., and Wu, W.~B. (2014).
\newblock Koml\'os-{{Major-Tusn\'ady}} approximation under dependence.
\newblock {\em Annals of Probability}, 42(2):794--817.

\bibitem[Billingsley, 1999]{Billingsley1999}
Billingsley, P. (1999).
\newblock {\em Convergence of {{Probability Measures}}}.
\newblock {John Wiley \textbackslash\& Sons}.

\bibitem[Bonis, 2020]{bonis2020}
Bonis, T. (2020).
\newblock Stein's method for normal approximation in {{Wasserstein}} distances
  with application to the multivariate central limit theorem.
\newblock {\em Probability Theory and Related Fields}, 178:827--860.

\bibitem[Breiman, 1967]{breiman1967}
Breiman, L. (1967).
\newblock On the tail behavior of sums of independent random variables.
\newblock {\em Zeitschrift f\"ur Wahrscheinlichkeitstheorie und Verwandte
  Gebiete}, 9(1):20--25.

\bibitem[Carlstein, 1986]{Carlstein1986}
Carlstein, E. (1986).
\newblock The {{Use}} of {{Subseries Values}} for {{Estimating}} the
  {{Variance}} of a {{General Statistic}} from a {{Stationary Sequence}}.
\newblock {\em The Annals of Statistics}, 14(3):1171--1179.

\bibitem[Chernozhukov et~al., 2013]{Chernozhukov2013}
Chernozhukov, V., Chetverikov, D., and Kato, K. (2013).
\newblock Gaussian approximations and multiplier bootstrap for maxima of sums
  of high-dimensional random vectors.
\newblock {\em Annals of Statistics}, 41(6):2786--2819.

\bibitem[Chernozhukov et~al., 2017]{Chernozhukov2017}
Chernozhukov, V., Chetverikov, D., and Kato, K. (2017).
\newblock Central limit theorems and bootstrap in high dimensions.
\newblock {\em Annals of Probability}, 45(4):2309--2352.

\bibitem[Chernozhukov et~al., 2019]{Chernozhukov2019}
Chernozhukov, V., Chetverikov, D., and Kato, K. (2019).
\newblock Inference on {{Causal}} and {{Structural Parameters}} using {{Many
  Moment Inequalities}}.
\newblock {\em The Review of Economic Studies}, 86(5):1867--1900.

\bibitem[Chernozhukov et~al., 2020]{chernozhukov2020}
Chernozhukov, V., Chetverikov, D., and Koike, Y. (2020).
\newblock Nearly optimal central limit theorem and bootstrap approximations in
  high dimensions.
\newblock {\em arXiv:2012.09513}.

\bibitem[Cs{\"o}rgo and R{\'e}v{\'e}sz, 1975]{csorgo1975}
Cs{\"o}rgo, M. and R{\'e}v{\'e}sz, P. (1975).
\newblock A new method to prove strassen type laws of invariance principle.
\newblock {\em Zeitschrift f\"ur Wahrscheinlichkeitstheorie und Verwandte
  Gebiete}, 31(4):255--259.

\bibitem[Dahlhaus, 1997]{Dahlhaus1997}
Dahlhaus, R. (1997).
\newblock Fitting time series models to nonstationary processes.
\newblock {\em Annals of Statistics}, 25(1):1--37.

\bibitem[Dahlhaus et~al., 2019]{Dahlhaus2017}
Dahlhaus, R., Richter, S., and Wu, W.~B. (2019).
\newblock Towards a general theory for nonlinear locally stationary processes.
\newblock {\em Bernoulli}, 25(2):1013--1044.

\bibitem[Einmahl, 1989]{einmahl1989}
Einmahl, U. (1989).
\newblock Extensions of results of {{Koml\'os}}, {{Major}}, and {{Tusn\'ady}}
  to the multivariate case.
\newblock {\em Journal of Multivariate Analysis}, 28(1):20--68.

\bibitem[Eldan et~al., 2020]{eldan2020}
Eldan, R., Mikulincer, D., and Zhai, A. (2020).
\newblock The {{CLT}} in high dimensions: {{Quantitative}} bounds via
  martingale embedding.
\newblock {\em The Annals of Probability}, 48(5):2494--2524.

\bibitem[Karatzas and Shreve, 1998]{Karatzas1998}
Karatzas, I. and Shreve, S.~E. (1998).
\newblock {\em Brownian {{Motion}} and {{Stochastic Calculus}}}, volume 113.
\newblock {Springer}, {New York}.

\bibitem[Karmakar and Wu, 2020]{Karmakar2020}
Karmakar, S. and Wu, W.~B. (2020).
\newblock Optimal {{Gaussian Approximation For Multiple Time Series}}.
\newblock {\em Statistica Sinica}, 30:1399--1417.

\bibitem[Koml{\'o}s et~al., 1975]{komlos1975}
Koml{\'o}s, J., Major, P., and Tusn{\'a}dy, G. (1975).
\newblock An approximation of partial sums of independent {{RV}}'-s, and the
  sample {{DF}}. {{I}}.
\newblock {\em Zeitschrift f\"ur Wahrscheinlichkeitstheorie und Verwandte
  Gebiete}, 32(1-2):111--131.

\bibitem[Koml{\'o}s et~al., 1976]{komlos1976}
Koml{\'o}s, J., Major, P., and Tusn{\'a}dy, G. (1976).
\newblock An approximation of partial sums of independent {{R}}.{{V}}.`s and
  the sample {{DF}}. {{II}}.
\newblock {\em Zeitschrift f\"ur Wahrscheinlichkeitstheorie und Verwandte
  Gebiete}, 34:33--58.

\bibitem[Kurisu et~al., 2021]{kurisu2021}
Kurisu, D., Kato, K., and Shao, X. (2021).
\newblock Gaussian approximation and spatially dependent wild bootstrap for
  high-dimensional spatial data.
\newblock {\em arXiv:2103.10720}.

\bibitem[Liu and Lin, 2009]{liu2009}
Liu, W. and Lin, Z. (2009).
\newblock Strong approximation for a class of stationary processes.
\newblock {\em Stochastic Processes and their Applications}, 119:249--280.

\bibitem[Liu et~al., 2013]{Liu2013}
Liu, W., Xiao, H., and Wu, W.~B. (2013).
\newblock Probability and moment inequalities under dependence.
\newblock {\em Statistica Sinica}, 23:1257--1272.

\bibitem[Mies, 2021]{mies2021}
Mies, F. (2021).
\newblock Functional estimation and change detection for nonstationary time
  series.
\newblock {\em Journal of the American Statistical Association}, to appear.

\bibitem[Peligrad and Shao, 1995]{Peligrad1995}
Peligrad, M. and Shao, Q. (1995).
\newblock Estimation of the {{Variance}} of {{Partial Sums}} for
  {$\rho$}-{{Mixing Random Variables}}.
\newblock {\em Journal of Multivariate Analysis}, 52(1):140--157.

\bibitem[Pinelis, 1994]{Pinelis1994}
Pinelis, I. (1994).
\newblock Optimum {{Bounds}} for the {{Distributions}} of {{Martingales}} in
  {{Banach Spaces}}.
\newblock {\em The Annals of Probability}, 22(4):1679--1706.

\bibitem[Wu, 2005]{Wu2005}
Wu, W.~B. (2005).
\newblock Nonlinear system theory: {{Another}} look at dependence.
\newblock {\em Proceedings of the National Academy of Sciences},
  102(40):14150--14154.

\bibitem[Wu and Zhou, 2011]{Wu2011}
Wu, W.~B. and Zhou, Z. (2011).
\newblock Gaussian approximations for non-stationary multiple time series.
\newblock {\em Statistica Sinica}, 21(3):1397--1413.

\bibitem[Zaitsev, 2007]{zaitsev2007}
Zaitsev, A.~Y. (2007).
\newblock Estimates for the rate of strong approximation in the
  multidimensional invariance principle.
\newblock {\em Journal of Mathematical Sciences}, 145(2):4856--4865.

\bibitem[Zaitsev, 2013]{Zaitsev2013}
Zaitsev, A.~Y. (2013).
\newblock The accuracy of strong {{Gaussian}} approximation for sums of
  independent random vectors.
\newblock {\em Russian Mathematical Surveys}, 68(4):721--761.

\bibitem[Zhai, 2018]{Zhai2018}
Zhai, A. (2018).
\newblock A high-dimensional {{CLT}} in {{W2}} distance with near optimal
  convergence rate.
\newblock {\em Probability Theory and Related Fields}, 170(3-4):821--845.

\bibitem[Zhang and Wu, 2017]{Zhang2017gauss}
Zhang, D. and Wu, W.~B. (2017).
\newblock Gaussian approximation for high dimensional time series.
\newblock {\em Annals of Statistics}, 45(5):1895--1919.

\bibitem[Zhang and Cheng, 2018]{Zhang2018}
Zhang, X. and Cheng, G. (2018).
\newblock Gaussian approximation for high dimensional vector under physical
  dependence.
\newblock {\em Bernoulli}, 24(4A):2640--2675.

\bibitem[Zhou, 2013]{Zhou2013}
Zhou, Z. (2013).
\newblock Heteroscedasticity and autocorrelation robust structural change
  detection.
\newblock {\em Journal of the American Statistical Association},
  108(502):726--740.

\bibitem[Zhou and Wu, 2009]{Zhou2009}
Zhou, Z. and Wu, W.~B. (2009).
\newblock Local linear quantile estimation for nonstationary time series.
\newblock {\em Annals of Statistics}, 37(5 B):2696--2729.

\end{thebibliography}
\bibliographystyle{apalike}

\end{document}